\documentclass[11pt]{article}
\usepackage[top=1in, bottom=1in, left=1.25in, right=1.25in, marginparwidth=1in, marginparsep=0.1in]{geometry}
\usepackage{amsmath,amssymb,amsthm,xcolor,enumerate}
\usepackage{esint} %
\usepackage[unicode,breaklinks=true,colorlinks=true,citecolor = {magenta}]{hyperref}

\DeclareFontFamily{U}{mathx}{\hyphenchar\font45}
\DeclareFontShape{U}{mathx}{m}{n}{
      <5> <6> <7> <8> <9> <10>
      <10.95> <12> <14.4> <17.28> <20.74> <24.88>
      mathx10
      }{}
\DeclareSymbolFont{mathx}{U}{mathx}{m}{n}
\DeclareFontSubstitution{U}{mathx}{m}{n}
\DeclareMathAccent{\widecheck}{0}{mathx}{"71}
\numberwithin{equation}{section}
\newtheorem{theorem}{Theorem}[section]
\newtheorem{corollary}[theorem]{Corollary}
\newtheorem{lemma}[theorem]{Lemma}
\newtheorem{proposition}[theorem]{Proposition}
\newtheorem{defn}[theorem]{Definition}

\theoremstyle{remark}
\newtheorem{remark}{Remark}[section]
\theoremstyle{definition}
\newtheorem{example}[theorem]{Example}
\newcommand{\bke}[1]{\left ( #1 \right )}

\newcommand{\bket}[1]{\left \{ #1 \right \}}
\newcommand{\norm}[1]{\left  \| #1  \right \|}

\newcommand{\abs}[1]{\left | #1 \right |}

\newcommand\al{\alpha}
\newcommand\be{\beta}

\newcommand\de{\delta}
\newcommand\ep{\epsilon}

\newcommand\e {\varepsilon}  %

\renewcommand\th{\theta}

\newcommand\si{\sigma}
\newcommand\ph{\varphi}

\newcommand\De{\Delta}

\newcommand\Si{\Sigma}
\newcommand\Om{\Omega}
\newcommand{\R}{\mathbb{R}}

\newcommand{\ZZ}{\mathbb{Z}}
\newcommand{\NN}{\mathbb{N}}
\newcommand{\cE}{\mathcal{E}}
\newcommand{\cF}{\mathcal{F}}
\newcommand{\cJ}{\mathcal{J}_\ep}

\newcommand{\cp}{\widecheck{p}}
\newcommand{\hp}{\widehat{p}}

\renewcommand{\div}{\mathop{\rm div}}

\newcommand{\supp} {\mathop{\mathrm{supp}}}
\newcommand{\pv} {\mathop{\mathrm{p.v.\!}}}
\newcommand{\tr} {\mathop{\mathrm{tr}}}

\newcommand{\pd}{\partial}
\newcommand{\pa}{\partial}

\newcommand{\nb}{\nabla}
\newcommand{\na}{\nabla}
\newcommand{\td}{\tilde}
\newcommand{\lec}{ \lesssim  }

\newcommand{\I}{\infty}

\newcommand{\EQ}[1]{\begin{equation}\begin{split} #1 \end{split}\end{equation}}
\newcommand{\EQN}[1]{\begin{equation*}\begin{split} #1 \end{split}\end{equation*}}

\newcommand{\uloc}{{\mathrm{uloc}}}
\newcommand{\loc}{{\mathrm{loc}}}

\begin{document}
\title{Global Navier-Stokes flows for non-decaying initial data with slowly decaying oscillation}

\author{Hyunju Kwon \and Tai-Peng Tsai}
\date{}

\maketitle
\begin{abstract}
Consider the Cauchy problem of incompressible Navier-Stokes equations in $\R^3$ with uniformly locally square integrable initial data. If the square integral of the initial datum on a ball vanishes as the ball goes to infinity, the existence of a time-global weak solution has been known. However, such data do not include constants, and the only known global solutions for non-decaying data are either for perturbations of constants, or when the velocity gradients are in $L^p$ with finite $p$. In this paper, we construct global weak solutions for non-decaying initial data whose local oscillations decay, no matter how slowly.



\end{abstract}
\section{Introduction}

In this paper, we consider the incompressible Navier-Stokes equations 
\begin{equation}\label{NS}\tag{NS}
\begin{cases}
\pa_t v -\De v + (v\cdot \na )v + \na p = 0   \\ 
\div v =0 \\
v|_{t=0}=v_0  
\end{cases}
\end{equation}
in $\R^3\times (0,T)$ for $0 <T\le \infty$. These equations describe the flow of incompressible viscous fluids, so the solution $v:\R^3\times (0,T)\to \R^3$ and $p:\R^3\times (0,T)\to \R$ represent the flow velocity and the pressure, respectively.
 
For an initial datum with finite kinetic energy, $v_0\in L^2(\R^3)$, the existence of a time-global weak solution dates back to Leray \cite{leray}. This solution has a finite global energy, i.e, it satisfies the energy inequality:
\EQ{\label{energy.ineq}
\norm{v(\cdot, t)}_{L^2(\R^3)}^2 + 2\norm{\na v}_{L^2(0,t;L^2(\R^3))}^2 
\leq \norm{v_0}_{L^2(\R^3)}^2, \quad \forall t>0. 
}
In Hopf \cite{Hopf}, this result is extended to smooth bounded domains with the Dirichlet boundary condition. We say $v$ is \textit{a Leray-Hopf weak solution} to \eqref{NS} in $\Omega \times (0,T)$ for a domain $\Om\subset \R^3$, if 
\[
v\in L^\infty(0,T;L^2_{\si}(\Om))\cap L^2(0,T;H_{0,\si}^1(\Om))\cap C_{wk}([0,T);L^2_{\si}(\Om))
\]
satisfies the weak form of \eqref{NS} and the energy inequality \eqref{energy.ineq}.

However, when a fluid fills an unbounded domain, it is possible to have finite local energy but
infinite global energy. One such example is a fluid with constant velocity.  There are also many interesting non-decaying infinite energy flows like time-dependent spatially periodic flows (flows on torus) and \emph{two-and-a-half dimensional flows}; see \cite[Section 2.3.1]{MaBe} and \cite{Gallagher}.  Can we get global existence for such data? To analyze the motion of such fluids, one may consider the class $ L^2_\uloc$ for velocity field $v_0$ in $\R^3$ whose kinetic energy is uniformly locally bounded. Here, for $1\le q \le \infty$, we denote by $L^q_\uloc$ the space of functions in $\R^3$ with
\[
\norm{v_0}_{L^q_\uloc} := \sup_{x_0 \in \R^3} \norm{v_0}_{L^q(B(x_0,1))}
<\infty.
\]
We also denote its subsapce with spatial decay
\[
E^q = \big\{v_0 \in L^q_{\uloc}: \, \lim_{|x_0|\to \infty}\norm{v_0}_{L^q(B(x_0,1))} =0\big\}.
\]
In \cite{LR}, Lemari\'{e}-Rieusset introduced the class of \textit{local energy solutions} for initial data $v_0 \in L^2_\uloc$ (see Section \ref{loc.ex.sec} for details). He proved the short time existence for initial data in $L^2_\uloc$, and the global in time existence for $v_0\in E^2$, those initial data in $L^2_\uloc$ which further satisfy the
 spatial decay condition%
\EQ{\label{ini.E2}
\lim_{|x_0|\to \infty}\int_{B(x_0,1)} |v_0|^2 dx=0.
}
Then, Kikuchi-Seregin \cite{KS} added more details to the results in \cite{LR}, especially the careful treatment of the pressure. They also allowed a force term $g$ in \eqref{NS} which satisfies $\div g=0$ and %
\[
\lim_{|x_0|\to \infty}\int_0^T\!\int_{B(x_0,1)}|g(x,t)|^2 dxdt  =0, \quad \forall T>0. 
\] 
Recently, Maekawa-Miura-Prange \cite{MaMiPr} generalized this result to the half-space $\R^3_+$. The treatment of the pressure in \cite{MaMiPr} is even more complicated.

One key difficulty in the study of infinite energy solutions is the estimates of the pressure. While finite energy solutions have enough decay at spatial infinity and one may often get the pressure from the equation $p = (-\De)^{-1}\pa_i\pa_j(v_iv_j)$, this is not applicable to infinite energy solutions because of their slow (or no) spatial decay. 

To estimate the pressure, the definition of a local energy solution in \cite{KS} includes a locally-defined pressure decomposition near each point in $\R^3$, see condition (v) in Definition \ref{les}. (It is already in \cite{LR} but not part of the definition.) In \cite{JiaSverak-minimal}-\cite{JiaSverak}, on the other hand, Jia and \v Sver\'ak use a slightly different definition by replacing the decomposition condition by the spatial decay of the velocity  
\begin{equation}
\label{decay.condi.parR}
\lim_{|x_0|\to \infty}\int_0^{R^2}\int_{B(x_0,R)}|v(x,t)|^2 dxdt =0, \quad\forall R>0.
\end{equation}
Under the decay assumption \eqref{ini.E2} on initial data, these two definitions can be shown to be  equivalent; see  \cite{MaMiPr, KMT}.  However, for general non-decaying initial data, the decay condition $\eqref{decay.condi.parR}$ is not expected, while the decomposition condition still works. For this reason, we follow the definition of Kikuchi-Seregin \cite{KS} in this paper. 

A new feature in the study of infinite energy solutions with non-decaying initial data is the abundance of \emph{parasitic solutions},
\[
v(x,t) = f(t),\quad p(x,t) = -f'(t)\cdot x
\]
for a smooth vector function $f(t)$. They solve the Navier-Stokes equations with initial data $f(0)$. If we choose $f_1(t)\not = f_2(t)$ with $f_1(0)=f_2(0)$, the corresponding parasitic solutions give two different local energy solutions with the same initial data. Such solutions have non-decaying initial data, and can be shown to fail the pressure decomposition condition. More generally, if $(v,p)$ is a solution to \eqref{NS}, then the following \emph{parasitic transform}
\EQ{
u(x,t) = v(y,t)+q'(t), \quad \pi(x,t)= p(y,t) - q''(t) \cdot y, \quad
 y=x-q(t)
}
gives another solution $(u,\pi)$ to \eqref{NS} with the same initial data $v_0$ for any vector function $q(t)$ satisfying $q(0)=q'(0)=0$.

We now summarize the known existence results in $\R^3$. 
In addition to the weak solution approach based on the a priori bound \eqref{energy.ineq} following Leray and Hopf, another fruitful approach is the theory of \emph{mild solutions}, treating the nonlinear term as a source term of the nonhomogeneous Stokes system. In the framework of  $L^q(\R^3)$, there exist short time mild solutions in $L^q(\R^3)$
when $3 \le q \le \infty$ (\cite{FJR,Kato84,GIM}).
When $q=3$, these solutions exist for all time for sufficiently small initial data in $L^3(\R^3)$; see \cite{Kato84}. Similar small data global existence results hold for many other spaces of similar scaling property, such as $L^3_{\text{weak}}$, Morrey spaces $M_{p,3-p}$, negative Besov spaces $\dot B^{3/q-1}_{q,\infty}$, $3<q<\infty$, and the Koch-Tataru space BMO$^{-1}$; See e.g.~\cite{GiMi,Kato,KoYa,Barraza,CP,BCD,Koch-Tataru}. 

For any data $v_0  \in L^q(\R^3)$, $2<q<3$, 
Calder\'on \cite{Calderon} constructed a global solution. His strategy is to first decompose $v_0 = a_0+b_0$ with small $a_0 \in L^3(\R^3)$ and large $b_0\in L^2(\R^3)$. A solution is then obtained as $v=a+b$, where $a$ is a global small mild solution of \eqref{NS} in $L^3(\R^3)$ with $a(0)=a_0$, and $b$ is a global weak solution of the $a$-perturbed Navier-Stokes equations in the energy class
with $b(0)=b_0$.

This idea is then used by Lemari\'{e}-Rieusset \cite{LR} to construct global local energy solutions for $v_0 \in E^2$; also see Kikuchi-Seregin \cite{KS}.

We now summarize the known existence results for non-decaying initial data. For the local existence,  many mild solution existence theorems mentioned earlier allow non-decaying data. The most relevant to us are Giga-Inui-Matsui  \cite{GIM} for initial data in $L^\infty(\R^3)$ and $BUC(\R^3)$, and Maekawa-Terasawa \cite{MT} for initial data in the closure of ${\bigcup_{p>3}L^p_\uloc}$ in $L^3_\uloc$-norm, and any small initial data in $L^3_\uloc$. Smallness is needed for $L^3_\uloc$ data even for short time existence.

When it comes to the global existence for non-decaynig data, a solution theory for perturbations of constant vectors seems straightforward. Lemari\'{e}-Rieusset \cite[Theorem 1(C)]{LR-Morrey} constructed global weak solutions for $u_0$ in Morrey space $M^{2,1}$, which contains non-decaying functions, e.g.
\[
v_0(x) = \sum_{k \in \NN} \zeta(x-x_k)
\]
with $|x_k|\to \infty$ rapidly as $k \to \I$. Here $\zeta$ is any smooth divergence free vector field with compact support.

The only other result we are aware of is the recent paper Maremonti-Shimizu \cite{MaSe}, which proved the global existence of weak solutions for initial data $v_0$ in $L^\infty(\R^3)\cap \overline{C_0(\R^3)}^{\dot{W}^{1,q}}$, $3<q<\infty$. In particular, they assume $\nb v_0 \in L^q(\R^3)$. Their strategy is to decompose the solution $v = U + w$, $U=\sum_{k=1}^n v^k$, where $v^1$ solves the Stokes equations with the given initial data, and $v^{k+1}$, $k\geq 1$, solves the linearized Navier-Stokes equations with the force $f^k=-v^{k}\cdot\na v^k$ and homogeneous initial data.  The force $f^1 \in L^q(0,T;L^q(\R^3))$ thanks to the assumption on $v_0$. In each iteration, we get an additional decay of the force $f^k$. The perturbation $w$ is then solved in the framework of weak solutions. The paper \cite{MaSe} motivated this paper.

We now state our main theorem.
Denote the average of a function $v$ in a set $O\subset \R^3$ by $(v)_O = \frac 1{|O|}\int_O v(x)\, dx$.
We denote $w\in E^2_\si$ if $w\in E^2$ and $\div w=0$.

\begin{theorem}\label{global.ex}
For any vector field $v_0\in E^2_\si + L^3_\uloc$ satisfying $\div v_0=0$ and
\EQ{\label{ini.decay}
\lim_{|x_0|\to \infty}\int_{B(x_0,1)}| v_0- (v_0)_{B(x_0,1)}| dx =0,  
}
we can find a time-global local energy solution $(v,p)$ to the Navier-Stokes equations \eqref{NS} in $\R^3 \times (0,\infty)$, in the sense of Definition \ref{les}.
\end{theorem}

Our main assumption is the ``\emph{oscillation decay}'' condition \eqref{ini.decay}. Note that all $v_0\in L^2_\uloc$ satisfying \eqref{ini.E2} also satisfy \eqref{ini.decay}. Furthermore, for $v_0\in L^2_\uloc$, either $v_0 \in E^1$ or $\na v_0\in E^{1}$ implies the condition \eqref{ini.decay}. Recall $E^q$ for $1 \le q \le \infty$ is the space of functions in $L^q_\uloc$ whose $L^q$-norm in a ball $B_1(x_0)$ goes to zero as $|x_0|$ goes to infinity. In particular,  our result generalizes the global existence for decaying initial data $v_0\in E^2$ in \cite{LR} and \cite{KS}. It also extends \cite{MaSe} for $v_0 \in L^\infty$ and $\nb v_0 \in L^q$.

\begin{example}
Consider
\[
v_0=v_1+v_2, \quad
v_1 = \frac {(-x_2, \ x_1, \ 0)}{\sqrt{|x|^2+1}},\quad
v_2=\frac {(-x_2, \ x_1, \ 0)}{|x|^2+1} \sin \bke{(x^2+1)^{100}}.
\]
We have
$\div v_1=\div v_2 = 0$, $v_0 \not \in E^2$, 
$v_0$ satisfies the oscillation decay condition, and
\[
 \limsup_{|x_0|\to \infty} \int_{B_1(x_0)} |\nb v_0| =\infty.
 \]
 In particular, $v_0 \in L^\infty$ but $\nb v_0 \not \in L^q$ for any $q \le \infty$.
Moreover, $v_0$ is not a perturbation of constant,  although it converges to a constant along each direction.

\end{example}

The condition $v_0 \in E^2_\si+L^3_\uloc$ gives us more regularity on the nondecaying part of $v_0$. We do not know if it is necessary for the global existence, but it is essential for our proof, and enables us to prove that for small $t>0$,
\EQ{\label{eq1.5}
 \norm{w(t)\chi_R}_{L^2_\uloc} \lec (t^\frac 1{20} + \norm{w_0\chi_R}_{L^2_\uloc}),
}
where $\chi_R(x)$ is a cut-off function supported in $|x|>R$, we decompose $v_0=w_0+u_0$ with $w_0\in E^2_\si$ and $u_0 \in L^3_{\uloc}$, and $w(t) = v(t) -e^{t\De}u_0$ with $w(0)=w_0$. This estimate shows that $\norm{w(t)\chi_R}_{L^2_\uloc}$ vanishes as $ t\to 0_+$ and $R\to \infty$.

The idea of our proof is as follows. First, we construct a local energy solution in a short time. For $v_0 \in L^2_\uloc$, this is done in \cite{LR} but not in \cite{KS}. However, we use a slightly revised approximation scheme to make all statements about the pressure easy to verify. In our scheme, we not only mollify the non-linear term as in \cite{leray} and \cite{LR}, but also insert a cut-off function, so that the non-linear term $(v\cdot \na)v$ is replaced by $(\cJ(v)\cdot \na)(v\Phi_{\ep})$, where $\cJ$ is a mollification of scale $\ep$ and $\Phi_\ep$ is a radial bump function supported in the ball $B(0, 2\ep^{-1})$.

Once we have a local-in-time local energy solution, we need some smallness to extend the solution globally in time. To this end, we decompose the solution as $v=V+w$ where $V(t)=e^{t\De}u_0$ solves the heat equation. The main effort is to show that $w(t) \in E^2$ for all $t$ and $w(t) \in E^6$ for almost all $t$. The proof is similar to the decay estimates in \cite{LR,KS} and we try to do local energy estimate for $w \chi_R$. The background $V$ has no spatial decay, but we can show the decay of $\nb V(x,t)$ in $L^\infty(B_R^c\times (t_0,\infty))$ as $R\to \infty$ for any $t_0>0$. This decay is not uniform up to $t_0=0$ as $u_0$ is rather rough. We need a new decomposition formula of the pressure, so that in the intermediate regions we can show the decay of the pressure using the decay of $\nb V$. Because the decay of $\nb V$ is not up to $t_0=0$, we need to do the local energy estimate in the time interval $[t_0,T)$, $0<t_0\ll1$. This forces us to prove the estimate \eqref{eq1.5}, and the \emph{strong local energy inequality} for $w$ away from $t=0$. 

Once we have shown $w(t) \in E^6$ for almost all $t<T$, we can extend the solution as in \cite{LR} and \cite{KS}. However, we avoid using the strong-weak uniqueness as in \cite{LR,KS}, and choose to verify the definition of local energy solutions directly as in \cite{MaMiPr}.
       
The rest of the paper consists of the following sections. In Section \ref{pre}, we discuss the properties of the heat flow $e^{t\De}u_0$, especially the decay of its gradient at spatial infinity assuming \eqref{ini.decay}. In Section \ref{loc.ex.sec}, we recall the definition of local energy solutions as in \cite{KS} and use our revised approximation scheme to find a local energy solution local-in-time. In Section \ref{decay.est.sec}, we find a new pressure decomposition formula suitable of using the decay of $\nb V$, prove the estimate  \eqref{eq1.5} and the strong local energy inequality, and then do the local energy estimate of $w\chi_R$, which implies $w(t) \in E^6$ for almost all $t$. In Section \ref{global.sec}, we construct the desired time-global local energy solution. In Section \ref{sec6}, by a similar and easier proof, we additionally obtain perturbations of time-global solutions with no spatial oscillation decay.

\section{Notations and preliminaries}\label{pre}

\subsection{Notation}
Given two comparable quantities $X$ and $Y$, the inequality $X\lesssim Y$ stands for $X\leq C Y$ for some positive constant $C$. In a similar way, $\gtrsim$ denotes $\geq C$ for some $C>0$. We write $ X \sim Y$ if $X\lesssim Y$ and $Y\lesssim X$. Furthermore, in the case that a constant $C$ in $X\leq C Y$ depends on some quantities $Z_1$, $\cdots$, $Z_n$, we write $X\lesssim_{Z_1,\cdots,Z_n}Y$. The notations $\gtrsim_{Z_1,\cdots,Z_n}$ and $\sim_{Z_1,\cdots,Z_n}$ are similarly defined. 

For a point $x\in \R^3$ and a positive real number $r$, $B(x,r)$ is the Euclidean ball in $\R^3$ centered at $x$ with a radius $r$,
\[
B(x,r) =B_r(x)= \{y\in \R^3: |y-x|<r\}.
\] 
When $x=0$, we denote $B_r = B(0,r)$. For a point $x\in \R^3$ and $r>0$, we denote the open cube centered at $x$ with a side length $2r$ as
\[
Q(x,r)=Q_r(x) = \bket{ y \in \R^3: \max_{i=1,2,3} |y_i -x_{i}| < r}.
\]

We denote the mollification $\cJ(v) = v\ast \eta_\ep$, $\ep>0$, where the mollifier is $\eta_\ep(x) = \e^{-3} \eta\left(\frac x{\ep}\right)$ and $\eta$ is a fixed nonnegative radial bump function in $C_c^\infty(\R^3)$ supported in $B(0,1)$ satisfying %
$\int\eta\, dx =1$. %

Various test functions in this paper are defined by rescaling and translating a non-negative radially decreasing bump function $\Phi$ satisfying $\Phi = 1$ on $B(0,1)$ and $\supp(\Phi)\subset B(0, \frac 32)$. 

For $k \in \NN \cup \{0,\infty\}$, let 
$C^k_c(\R^3)$ be the subset of functions in $C^k(\R^3)$ with compact supports, and
\[
C^{k}_{c,\si}(\R^3) = \bket{ u \in C^{k}_{c}(\R^3;\R^3):\ \div u =0}.
\]

\subsection{Uniformly locally integrable spaces}

To consider infinite energy flows, we work in the spaces $L^q_\uloc$, $1\leq q\leq \infty$, and $U^{s,p}(t_0,t)$ for $1\leq s, p\leq \infty$ and $0\leq t_0<t\le \infty$, defined by
\[
L^q_\uloc = 
\bket{u\in L^1_{\loc}(\R^3): \norm{u}_{L^q_\uloc} = \sup_{x_0\in \R^3} \norm{u}_{L^q(B_1(x_0))} <+\infty  
}
\]
and 
\[
U^{s,p}(t_0,t) = 
\bket{u \in L^1_{\loc}(\R^3\times(t_0,t)): \norm{u}_{U^{s,p}(t_0,t)}= \sup_{x_0\in \R^3} \norm{u}_{L^s(t_0,t;L^p(B_1(x_0)))}<+\infty }.
\]
When $t_0=0$, we simply use $U^{s,p}_T = U^{s,p}(0,T)$. Note that $U^{\infty,p}(t_0,t)=L^\infty(t_0,t;L^p_\uloc)$, $1\leq p \leq \infty$,
but for general $1\leq s<\infty$ and $1\leq p \le \infty$, $U^{s,p}(t_0,t)$ and $L^s(t_0,t;L^p_\uloc)$ are not equivalent norms. Indeed, we can only guarantee that
\EQ{\label{Usp.le.LsUp}
\norm{u}_{U^{s,p}(t_0,t)} \leq \norm{u}_{L^s(t_0,t;L^p_\uloc)},
}
but not the inequality of the other direction. 

\medskip
\begin{example}
Fix $1\le s <\infty$ and $p \in [1,\infty]$. Let $x_k$ be a sequence in $\R^3$ with disjoint $B_1(x_k)$, $k \in \NN$, and let $t_k = t_0+2^{-k}$. Define a function $u$ by $u(x,\tau)=2^{k/s}$ on $B_1(x_k) \times (t_0,t_k)$, $k \in \NN$, and $u(x,\tau)=0$ otherwise. It is defined independently of $p$. We have $u \in U^{s,p}(t_0,t)$, but 
\[
\int _{t_0}^{t_1} \norm{u(\cdot,\tau)}_{L^p_\uloc}^s d\tau = \sum_{k=1}^\infty \int_{t_{k+1}}^{t_k} c_p2^k d\tau = \sum_{k=1}^\infty \frac 12 c_p=\infty,
\]
and hence $u \not \in L^s(t_0,t;L^p_\uloc)$.
\qed
\end{example}
\medskip

We define a local energy space $\cE(t_0,t)$ by
\EQ{\label{cE.def}
\cE(t_0,t) = \bket{u\in L^2_\loc([t_0,t]\times \R^3;\R^3) \ : \ \div u =0, \ 
\norm{u}_{\cE(t_0,t)} <+\infty},
}
where
\[
\norm{u}_{\cE(t_0,t)} := \norm{u}_{U^{\infty,2}(t_0,t)} + \norm{\na u}_{U^{2,2}(t_0,t)}.
\]
When $t_0=0$, we use the abbreviation $\cE_T = \cE(0,T)$. 

The spaces $E^p$ and $G^p(t_0,t)$, $1\leq p\leq \infty$, are defined by an additional decay condition at infinity,
\[
E^p:= \{f\in L^p_\uloc: \norm{f}_{L^p(B(x_0,1))} \to 0, \quad\text{as }|x_0|\to \infty \}, 
\]
and
\[
G^p(t_0,t) :=\{u \in U^{p,p}(t_0,t):  \ \norm{u}_{L^p([t_0,t]\times B(x_0,1))} \to 0, \quad\text{as } |x_0|\to \infty
\}.
\]
We let $L^p_{\uloc,\si}$,  $E^p_\si$ and $G^p_\si(t_0,t)$ denote divergence-free vector fields with components in $L^p_{\uloc}$, $E^p$ and $G^p(t_0,t)$, respectively.

The space $E^p$, $1\leq p<\infty$, can be characterized as $\overline{C_{c}^\infty(\R^3)}^{L^p_\uloc}$.
The analogous statement for $E^p_\si$ is true.
\begin{lemma}\textup{(\cite[Appendix]{KS})}\label{decomp.Ep}
Suppose that $f\in E^p_\si$ for some $1\leq p<\infty$. Then, for any $\ep>0$, we can find $f^\ep \in C_{c,\si}^\infty(\R^3)$ such that
\[
\norm{f-f^\ep}_{L^p_\uloc}<\ep. 
\]
\end{lemma}

\subsection{Heat and Oseen kernels on $L^q_\uloc$}

Now, we study the operators $e^{t\De}$ and $e^{t\De}\mathbb{P}\na \cdot$ on $L^q_\uloc$. Here $\mathbb{P}$ denotes the Helmholtz projection in $\R^3$. Both are defined as convolution operators
\[
e^{t\De} f = H_t \ast f, \quad\text{and}\quad
e^{t\De}\mathbb{P}_{ij}\pa_kF_{jk} = \pa_k S_{ij}\ast F_{jk}, 
\]
where $H_t$ and $S_{ij}$ are the heat kernel and the Oseen tensor, respectively,
\EQN{
H_t(x) = \frac 1{\sqrt{4\pi t}^3} \exp\left(-\frac {|x|^2}{4t}\right),
}
and
\[
S_{ij}(x,t) = H_t(x)\de_{ij}  + \frac 1{4\pi} \frac{\pa^2}{\pa_{x_i}\pa_{x_j}} \int_{\R^3} \frac{H_t(y)}{|x-y|}dy. 
\]
In this note, we use $(\div F)_i = (\na \cdot F)_i = \pa_j F_{ji}$. Note that the Oseen tensor satisfies the following pointwise estimates
\EQ{\label{pt.est.oseen}
|\na_{x}^l \pa_t^k S(x,t)| \leq C_{k,l} (|x|+\sqrt{t})^{-3-l-2k}.
}

We have the following estimates.

\begin{lemma}[Remark 3.2 in \cite{MT}]\label{lemma23} 
For $1\leq q\leq p\leq \infty$, the following holds. For any vector field $f$ and any 2-tensor $F$ in $\R^3$,
\[
\norm{\pa^{\al}_t \pa^{\be}_x e^{t\De}f}_{L^p_\uloc}
\lec \frac 1{t^{|\al|+\frac{|\be|}{2}}}\left(1+ \frac 1{t^{\frac 32\left(\frac 1q-\frac 1p\right)}}\right)\norm{f}_{L^q_\uloc},
\]

\[
\norm{\pa^{\al}_t \pa^{\be}_x e^{t\De}\mathbb{P}\na \cdot F}_{L^p_\uloc}
\lec \frac 1{t^{|\al|+\frac{|\be|}{2}+\frac 12}}\left(1+ \frac 1{t^{\frac 32\left(\frac 1q-\frac 1p\right)}}\right)\norm{F}_{L^q_\uloc}.
\]
\end{lemma}

Note $p=\infty$ is allowed, with $L^\infty_\uloc= L^\infty$.

\medskip

\begin{lemma}\label{lemma24} 
For any $T>0$, if $f\in L^2_\uloc$ and $F\in U^{2,2}_T$, then we have
\EQN{
\norm{e^{t\De}f}_{\cE_T}
\lec (1+T^\frac 12)\norm{f}_{L^2_\uloc},\\
\norm{\int_0^t e^{(t-s)\De}\mathbb{P}\na \cdot F(s)\, ds}_{\cE_T}\lec (1+T)\norm{F}_{U^{2,2}_T}.
}
\end{lemma}

Recall $\norm{u}_{\cE_T} = \norm{u}_{U^{\infty,2}_T} + \norm{\na u}_{U^{2,2}_T}$. Similar estimates can be found in the proof
of \cite[Theorem 14.1]{LR2}.
We give a slightly revised proof here for completeness.

\begin{proof} Fix $x_0\in \R^3$ and let $\phi_{x_0}(x) = \Phi\left(\frac{x-x_0}2\right)$. We decompose $f$ and $F$ as 
\[
f = f\phi_{x_0}+ f(1-\phi_{x_0}) = f_1 + f_2
\]
and
\[
F = F\phi_{x_0}+ F(1-\phi_{x_0}) = F_1 + F_2.
\]

Since $f_1\in L^2(\R^3)$ and $F_1\in L^2(0,T;L^2(\R^3))$, by the usual energy estimates for the heat equation and the Stokes system, we get
\EQ{\label{0708-1}
\norm{e^{t\De}f_1}_{\cE_T} \lec \norm{f_1}_{2} \lec\norm{f}_{L^2_\uloc}
}
and
\EQ{\label{0709-1}
\norm{\int_0^t e^{(t-s)\De}\mathbb{P}\na \cdot F_1(s) ds}_{\cE_T}\lec
\norm{F_1}_{L^2(0,T;L^2(\R^3))} \lec \norm{F}_{U^{2,2}_T}.
}

On the other hand, by Lemma \ref{lemma23},
\[
\norm{e^{t\De}f_2}_{U^{\infty,2}_T}=\norm{e^{t\De}f_2}_{L^\infty(0,T;L^2_\uloc)} \lec \norm{f_2}_{L^2_\uloc} \lec \norm{f}_{L^2_\uloc}.
\]
Together with \eqref{0708-1}, we get 
\EQ{\label{0708-2}
\norm{e^{t\De}f}_{U^{\infty,2}_T}\lec \norm{f}_{L^2_\uloc}.
}
(This also follows from Lemma \ref{lemma23}.)
By the heat kernel estimates,
\EQN{
\norm{\na e^{t\De}f_2}_{L^2((0,T)\times B(x_0,1))}
&\lec T^{\frac 12}\norm{\na e^{t\De}f_2}_{L^\infty((0,T)\times B(x_0,1))} \\
&\lec T^{\frac 12}\int_{B(x_0,2)^c}\frac 1{|x_0-y|^4} |f_2(y)|dy\\
&\le T^{\frac 12}\sum_{k=1}^\infty \int_{B(x_0,2^{k+1})\setminus B(x_0,2^{k})}\frac 1{|x_0-y|^4} |f(y)|dy\\
&\lec T^{\frac 12}\sum_{k=1}^\infty \frac 1{2^{4k}} \int_{B(x_0,2^{k+1})} |f(y)|dy.
}
We may cover $B(x_0,2^{k+1})$ by $\bigcup_{j=1}^{J_k} B(x_j^k,1)$ with 
$J_k$ bounded by $C_0 2^{3k}$ for some constant $C_0>0$. Then
\[
\norm{\na e^{t\De}f_2}_{L^2((0,T)\times B(x_0,1))}\lec  
T^{\frac 12}\sum_{k=1}^\infty \frac 1{2^{4k}} \sum_{j=1}^{J_k} \int_{B(x_j^k,1)} |f(y)|dy\\
\lec T^{\frac 12} \norm{f}_{L^2_\uloc} .
\] 
Together with \eqref{0708-1}, we get
\[
\norm{\na e^{t\De}f}_{L^2((0,T)\times B(x_0,1))}\lec (1+T^{\frac 12}) \norm{f}_{L^2_\uloc}.
\]
Taking supremum in $x_0$, we obtain
\[
\norm{\na e^{t\De}f}_{U^{2,2}_T} \lec (1+T^{\frac 12}) \norm{f}_{L^2_\uloc}.
\]
This and \eqref{0708-2} show the first bound of the lemma,
$\norm{e^{t\De}f}_{\cE_T}\lec (1+T^\frac 12)\norm{f}_{L^2_\uloc}$.

Denote $\Psi F (t) = \int_0^t e^{(t-s)\De}\mathbb{P}\na \cdot F(s) ds$. By the pointwise estimates \eqref{pt.est.oseen} for Oseen tensor, we have
\EQN{
\norm{\Psi F_2}_{L^\infty(0,T;L^2( B(x_0,1)))}
&\lec \int_0^t\! \int_{B(x_0,2)^c} \frac 1{|x_0-y|^4} |F_2(y,s)|dy ds\\
&\leq \sum_{k=1}^\infty \frac 1{2^{4k}}\int_0^t\! \int_{B(x_0,2^{k+1})} |F(y,s)|dy ds\\
&\leq \sum_{k=1}^\infty \frac 1{2^{4k}}\sum_{j=1}^{J_k}\int_0^t\! \int_{B(x_j^k,1)} |F(y,s)|dy ds\\
&\lec \norm{F}_{U^{1,1}_T} \lec T^\frac 12\norm{F}_{U^{2,2}_T}
}
and 
\EQN{
\norm{\na \Psi F_2}_{L^2((0,T)\times B(x_0,1))}
&\lec  T^\frac 12 \norm{\na \Psi F_2}_{{L^\infty((0,T)\times B(x_0,1))}} \\
&\lec T^\frac 12\int_0^t\! \int_{B(x_0,2)^c} \frac 1{|x_0-y|^5} |F_2(y,s)|dy ds\\
&\leq T^\frac 12\sum_{k=1}^\infty \frac 1{2^{5k}}\int_0^t\! \int_{B(x_0,2^{k+1})} |F(y,s)|dy ds\\
&\lec T\norm{F}_{U^{2,2}_T}.
}
Combined with \eqref{0709-1}, we have
\[
\norm{\Psi F}_{L^\infty(0,T;L^2( B(x_0,1)))}\lec (1+T^\frac 12)\norm{F}_{U^{2,2}_T}
\]
and
\[
\norm{\na \Psi F}_{L^2((0,T)\times B(x_0,1))}
\lec (1+T)\norm{F}_{U^{2,2}_T}.
\]
Finally, we take suprema in $x_0$ to get
\[
\norm{\int_0^t e^{(t-s)\De}\mathbb{P}\na \cdot F(s) ds}_{\cE_T}
\lec (1+T) \norm{F}_{U^{2,2}_T}. 
\]
This is the second bound of the lemma.          
\end{proof}

\subsection{Heat kernel on $L^1_\uloc$ with decaying oscillation}

In this subsection, we investigate how the decaying oscillation assumption \eqref{ini.decay} on initial data affects the heat flow. Recall
\[
(u)_{Q_r(x)} = \fint_{Q_r(x)} u(y)\, dy = \frac 1{|Q_r(x)|} \int_{Q_r(x)} u(y)\, dy.  
\]

\begin{lemma}\label{from.ini}
Suppose that $u\in L^1_{\uloc}(\R^3)$ satisfies
\EQ{\label{0505-1}
\lim_{|x_0|\to \infty} \int_{Q_1(x_0)} |u - (u)_{Q_1(x_0)}|dx =0.
}
Then, for any $r>0$, we have
\EQ{\label{0505-2}
\lim_{|x_0|\to \infty} \int_{Q_r(x_0)} |u - (u)_{Q_r(x_0)}| dx=0,
}
and
\EQ{\label{0505-3}
\lim_{|x_0|\to \infty} \sup_{y\in \overline{Q_{2r}(x_0)}} |(u)_{Q_r(y)} - (u)_{Q_r(x_0)}|  =0.
}
\end{lemma}

\begin{proof}
First note that $(u)_{Q_r(x)}$ is finite for any $x\in \R^3$ and $r>0$. Indeed, 
\[
|(u)_{Q_r(x)}| \le C_r\norm{u}_{L^1_\uloc}
\]
for a constant $C_r$ independent of $x$, $C_r < C$ for $r>1$, and $C_r \sim r^{-3}$ for $r \ll 1$.

Fix $x_0\in \R^3$ and $r>0$. For any constant $c\in \R$, we get
\EQN{
\fint_{Q_r(x_0)} |u - (u)_{Q_r(x_0)}| dx 
& \le \fint_{Q_r(x_0)} |u - c|  + |(u)_{Q_{r}(x_0)} - c| dx
\\
& = \fint_{Q_r(x_0)} |u - c| dx +  \abs{ \fint_{Q_r(x_0)} \bke{u - c} dx}
\\
& \le 2\fint_{Q_r(x_0)} |u - c| dx .
}

Then, for $Q_r =Q_r(x_1)\subset  Q_R(x_0)$, $R>r$, we get
\EQ{\label{0505-4}
\fint_{Q_r} |u - (u)_{Q_r}| dx \le 2\fint_{Q_r} |u - (u)_{Q_R(x_0)}|dx  \le  \frac{2R^3}{r^3} \fint_{Q_R(x_0)} |u - (u)_{Q_R(x_0)}| dx .
}

With $x_0=x_1$ and $R=1$ in \eqref{0505-4}, \eqref{0505-1} implies \eqref{0505-2} for all $r \in (0,1)$.

If $y \in \overline{Q_{2r}(x_0)}$, then
\[
Q_r(x_0) \cup Q_r(y) \subset Q_R(x_1), \quad x_1 = \frac 12(x_0+y), \quad R \ge2r.
\]
Thus, 
\EQ{\label{0505-7}
\abs{(u)_{Q_r(x_0)} - (u)_{Q_r(y)}} 
&\le \abs{ \fint_{Q_r(x_0)} u  -(u)_{Q_R(x_1)} dx} +\abs{ \fint_{Q_r(y)} u  -(u)_{Q_R(x_1)} dx}
\\
&\le  \fint_{Q_r(x_0)} |u - (u)_{Q_R(x_1)}| dx +  \fint_{Q_r(y)} |u - (u)_{Q_R(x_1)}| dx
\\
& \le  \frac{2R^3}{r^3} \fint_{Q_R(x_1)} |u - (u)_{Q_R(x_1)}| dx .
}
With $R=1$, this and \eqref{0505-1} imply \eqref{0505-3} for all $r \in (0,\frac12]$.

Now, for any $Q_r(x_0)$ with $r>1$, choose the smallest integer $N>2r$ and let $\rho = r/N< \frac 12$. We can find a set $S=S_{x_0,r}$ of $N^3$ points such that $\{ Q_\rho(z): z \in S \}$ are disjoint and
\[
\overline{  Q_r(x_0)} = \bigcup_{z \in S} \overline{Q_\rho(z)}.
\]
For any $z,z' \in S$, we can connect them by points $z_j$ in $S$, $j=0,1,\dots, N$, such that
$z_0=z$, $z_N=z'$, and
$z_{j} \in \overline{Q_{2\rho}(z_{j-1})}$, $j=1, \ldots , N$. We allow $z_{j+1}=z_j$ for some $j$. Thus 
\[
|(u)_{Q_\rho(z)} - (u)_{Q_\rho(z')}| \le \sum_{j=1}^{N}  |(u)_{Q_\rho(z_{j})} - (u)_{Q_\rho(z_{j-1})}|,  
\]
and hence
\EQ{\label{0505-5}
\max_{z,z'\in S_{x_0,r}}  |(u)_{Q_\rho(z)} - (u)_{Q_\rho(z')}| =o(1)\quad \text{as } |x_0| \to \infty
}
by \eqref{0505-3} as $\rho \in (0,\frac12)$.
We have
\EQN{
\fint_{Q_r(x_0)}& |u - (u)_{Q_r(x_0)}| dx 
\\
&= \sum_{z \in S} N^{-3} \fint_{Q_\rho(z)} |u - (u)_{Q_r(x_0)}| dx
\\
&\le \sum_{z \in S} N^{-3}\bke{ \fint_{Q_\rho(z)} |u - (u)_{Q_\rho(z)}| + |(u)_{Q_r(x_0)} - (u)_{Q_\rho(z)}| dx }
\\
&\le \bke{\sum_{z \in S} N^{-3} \fint_{Q_\rho(z)} |u - (u)_{Q_\rho(z)}| dx} + \max_{z,z'\in S}  |(u)_{Q_\rho(z)} - (u)_{Q_\rho(z')}| 
\\
&=o(1)\qquad \text{as } |x_0| \to \infty
}
by \eqref{0505-2} and \eqref{0505-5} for $\rho \in (0,\frac12)$. This shows 
\eqref{0505-2} for all $r>1$.

Finally, \eqref{0505-3} for $r>1/2$ follows from \eqref{0505-2} and \eqref{0505-7}.
\end{proof}

The following lemma says that decaying oscillation over \emph{cubes} is equivalent with decaying oscillation over \emph{balls}.
\begin{lemma}
Suppose $u \in L^1_\uloc$. Then $u$ satisfies \eqref{0505-1} if and only if
\EQ{\label{0607-1}
\lim_{|x_0|\to \infty} \int_{B_1(x_0)} |u - (u)_{B_1(x_0)}|dx =0.
}
\end{lemma}
\begin{proof}
Let $\rho = 3^{-1/2}$. We have $Q_\rho(x_0) \subset B_1(x_0) \subset Q_1(x_0)$. Similar to the proof of \eqref{0505-4}, we have
\[
\int_{B_1(x_0)} |u - (u)_{B_1(x_0)}|dx \le C\int_{Q_1(x_0)} |u - (u)_{Q_1(x_0)}|dx
\]
and hence \eqref{0607-1} follows from \eqref{0505-1}. Similarly, we also have
\[
\int_{Q_\rho(x_0)} |u - (u)_{Q_\rho(x_0)}|dx \le C
\int_{B_1(x_0)} |u - (u)_{B_1(x_0)}|dx 
\]
and hence \eqref{0505-2} for $r = \rho$ follows from \eqref{0607-1}. 
Then $v(x)=u(\rho x)$ satisfies  \eqref{0505-1}.
By Lemma \ref{from.ini}, $v$ satisfies \eqref{0505-2} for any $r>0$, and we get  \eqref{0505-1} for $u$.
\end{proof}

\begin{lemma}\label{decay.na.U}
Suppose $v_0\in L^1_\uloc$ and
\[
\int_{Q(x_0,1)}|v_0-(v_0)_{Q(x_0,1)}|\to 0, \quad\text{ as } |x_0|\to \infty.
\]
Let $V=e^{t\De}v_0$. 
Then $(\na V)(t_0)\in C_0(\R^3)$ for every $t_0>0$. Furthermore, for any $t_0>0$, we have
\EQ{\label{decay.naV}
\sup_{t>t_0}\norm{\na V(\cdot, t)}_{L^\infty(B(x_0,1))} \to 0, \quad\text{ as } |x_0|\to \infty. 
}
\end{lemma}

\begin{proof}
For $k \in \ZZ^3$, let $\Si_k$ denote the set of its neighbor integer points,
\[
\Si_k = \ZZ^3 \cap Q(k,1.01)\setminus \{ k\}.
\]
Let
\[
a_k = (v_0)_{Q_1(k)}, \quad b_k = \max _{ k' \in \Si_k} |a_{k'}-a_k|,
\quad
c_k =  \int_{ Q_1(k)} |v_0(x)-a_k| dx.
\]
By the assumption, $c_k\to 0$ as $|k|\to \infty$ and by Lemma \ref{from.ini}, $b_k\to 0$ as $|k|\to \infty$. 

Choose a nonnegative $\phi \in C^\infty_c(\R^3)$ with $\supp \phi \subset Q_1(0)$ and 
\[
\sum_{k \in \ZZ^3}  \phi_k (x)=1 \quad \forall x \in \R^3, \quad \phi_k(x) = \phi(x-k).
\]
Define 
\[
v_1(x) = \sum_{k \in \ZZ^3} a_k \phi_k(x).
\]
Since $|a_k|\lec \norm{v_0}_{L^1_\uloc}$, $v_1$ is in $L^\infty(\R^3)$. For $x\in Q_1(k)$, it can be written as
\[
v_1(x)  = a_k +  \sum_{k' \in \Si_k} (a_{k'} -a_k)\phi_{k'}(x).
\]
Thus
\EQ{\label{diff.v01}
\int_{ Q_1(k)} |v_0(x)-v_1(x)| dx &\le \int_{ Q_1(k)} |v_0(x)-a_k| dx + \sum_{k' \in \Si_k}  \int_{ Q_1(k)} |a_k  -a_{k'}|\phi_{k'}(x) dx 
\\
& \le c_k +C b_k ,
}
and
\EQ{\label{nb.v1}
\sup_{x \in Q_1(k)} |\nb v_1(x)| 
\le \sup_{x \in Q_1(k)}  \sum_{k' \in \Si_k} |a_{k'} -a_k| \cdot |\nb \phi_{k'}(x)|
\le Cb_{k} .
}

Let $\psi_R(x) = \Phi\left(\frac{x}{R}\right)$. %
We decompose
\EQN{
\nb V(x,t) &= \int \nb H_t(x-y) v_0(y) (1-\psi_R(x-y))dy 
\\
&\quad+ \int \nb H_t(x-y) [v_0(y)-v_1(y)] \psi_R(x-y) dy
\\
&\quad+ \int \nb H_t(x-y) v_1(y) \psi_R(x-y) dy = I_1+ I_2 +I_3.
}
By integration by parts, we can rewrite $I_3$,
\[
I_3 = \int  H_t(x-y) \nb v_1(y) \psi_R(x-y) dy
-\int  H_t(x-y) v_1(y) (\nb\psi_R)(x-y) dy = I_{31} + I_{32}.
\]

Fix $\ep>0$ and consider $t>t_0>0$. 
Since for any $t>0$ and $x\in \R^3$, we have
\EQN{
|I_1| &\lec \int_{B(x,R)^c} \frac {|x-y|^5}{t^{\frac 52}}e^{-\frac {|x-y|^2}{4t}} \frac 1{|x-y|^4}|v_0(y)| dy \\
&\lec \int_{B(x,R)^c} \frac 1{|x-y|^4}|v_0(y)| dy \lec \frac 1R\norm{v_0}_{L^1_\uloc},
}
and
\EQN{
|I_{32}| \lec \norm{H_t}_1\norm{v_1}_\infty \norm{\na \psi_R}_\infty 
\lec \frac 1R\norm{v_0}_{L^1_\uloc},
}
we can choose sufficiently large $R>0$ such that
\[
|I_1, I_{32}|<\ep.
\]

The integrands of both $I_2$ and $I_{31}$ are supported in  $|y-x|\leq 2R$. If $|x|>2\rho$ with $\rho >2R$ and 
$|y-x|\leq 2R$,  then 
$|y|\geq |x|-|x-y| > \rho$. Let 
\[
1_{>\rho}(y) = 1\quad \text{for} \quad |y|>\rho,\quad \text{and}\quad 1_{>\rho}(y) = 0\quad \text{for} \quad |y|\leq \rho.
\]
We have
\EQN{
|I_2| \leq \norm{ |\na H_t|\ast |v_0-v_1|1_{>\rho}}_{L^\infty(\R^3)}
\lec {t_0^{-\frac 12}}\left(1+{t_0^{-\frac 32}}\right)\norm{|v_0-v_1|1_{>\rho}}_{L^1_\uloc}
}
by Lemma \ref{lemma23}, and 
\[
|I_{31}| \leq \norm{e^{t\De}(|\na v_1|1_{>\rho})}_{L^\infty(\R^3)}
\lec \norm{|\na v_1|1_{>\rho}}_{L^\infty(\R^3)} .  
\]
If we take $\rho$ sufficiently large, 
by \eqref{diff.v01} and \eqref{nb.v1}, we have $|I_2|+ |I_{31}| \leq 2 \ep$.

Since for any $t>t_0$ and $\ep>0$, we can choose $\rho>0$ such that
\[
\sup_{t>t_0}\norm{\na V(\cdot,t)}_{L^\infty(B(0, 2\rho)^c)}<4\ep, 
\]
we get \eqref{decay.naV}.
\end{proof}

\section{Local existence}\label{loc.ex.sec}
In this section, we recall the definition of local energy solutions and prove their \emph{time-local} existence using a revised approximation scheme. Note that we do not assume spatial decay of initial data for the time-local existence.

As mentioned in the introduction, we follow the definition in Kikuchi-Seregin \cite{KS}. 

\begin{defn}[local energy solution]\label{les}
Let $v_0 \in L^2_\uloc$ with $\div v_0=0$.
A pair $(v,p)$ of functions is a local energy solution to the Navier-Stokes equations \eqref{NS}  with initial data $v_0$
in $\R^3\times (0,T)$, $0 <T< \infty$, if it satisfies the followings.
\begin{enumerate}[(i)]
\item $v\in \cE_{T}$, defined in \eqref{cE.def}, and $p \in L^\frac 32_\loc([0,T)\times \R^3)$.
\item $(v,p)$ solves the Navier-Stokes equations \eqref{NS} in the distributional sense.
\item For any compactly supported function $\ph\in L^2(\R^3)$, the function $\int_{\R^3} v(x,t)\cdot \ph(x)\, dx$ of time is continuous on $[0,T]$. Furthermore, for any compact set $K\subset \R^3$, 
\[
\norm{v(\cdot,t)-v_0}_{L^2(K)} \rightarrow 0, \quad \text{as }t\to 0^+.
\]

\item $(v,p)$ satisfies the local energy inequality (LEI) for any $t\in(0,T)$:
\EQ{\label{LEI}
\int_{\R^3} & |v|^2\xi (x,t)dx + 2\int_0^t\! \int_{\R^3}|\na v|^2\xi \,dxds\\
&\quad \leq \int_0^t\! \int_{\R^3} |v|^2(\pa_s\xi + \De \xi) + (|v|^2+2p)(v\cdot \na)\xi \,dxds,  
}
for all non-negative smooth functions $\xi\in C^\infty_c((0,T)\times\R^3)$. 

\item For each $x_0\in \R^3$, we can find $c_{x_0}\in L^\frac 32(0,T)$ such that
\EQ{\label{pressure.decomp}
p(x,t)=\hp_{x_0}(x,t)+c_{x_0}(t), 
\qquad\text{in } L^{\frac 32}(B(x_0, \tfrac32)\times(0,T)),
}
where
\EQ{\label{hp.def}
\hp_{x_0}(x,t) =& -\frac 13 |v(x,t)|^2 + \pv \int_{B(x_0,2)} K_{ij}(x-y)v_iv_j(y,t)dy\\
&+\int_{B(x_0,2)^c} (K_{ij}(x-y)-K_{ij}(x_0-y))v_iv_j(y,t)dy
}
for $K(x)=\frac 1{4\pi|x|}$ and $K_{ij} = \pa_{ij}K$. 
\end{enumerate}
We say the pair $(v,p)$ is a local energy solution to \eqref{NS} in $\R^3\times (0,\infty)$ if it is a local energy solution to \eqref{NS} in $\R^3\times (0,T)$ for all $0 <T< \infty$.\qed 
\end{defn}

For an initial data $v_0\in L^2_\uloc$ whose local kinetic energy is uniformly bounded, we reprove the local existence of a local energy solution of \cite[Chapt 32]{LR}.

\begin{theorem}[Local existence]\label{loc.ex}
Let $v_0\in L^2_\uloc$ with $\div v_0 =0$. If
\[
T\le \frac{\e_1}{1+\norm{v_0}_{L^2_\uloc}^4}
\]
for some small constant $\e_1>0$,
we can find a local energy solution $(v,p)$ on $\R^3\times (0,T)$ to Navier-Stokes equations \eqref{NS} for the initial data $v_0$, satisfying $\norm{v}_{\cE_T} \le C \norm{v_0}_{L^2_\uloc}$.
\end{theorem}

Note that we do not assume $v_0 \in E^2$, i.e., we do not assume spatial decay of $v_0$. Although the local existence theorem is proved in \cite[Chapt 32]{LR}, a few details are missing there, in particular those related to the pressure. These details are given in \cite{KS} for the case $v_0 \in E^2$. Here we treat the general case $v_0 \in L^2_\uloc$.

Recall the definitions of $\cJ(\cdot)$ and $\Phi$ in Section \ref{pre} and let $\Phi_\ep(x)= \Phi (\ep x )$, $\ep>0$. To prove Theorem \ref{loc.ex}, we consider approximate solutions $(v^\ep, p^\ep)$ to the localized-mollified Navier-Stokes equations
\EQ{\label{reg.NS}
\begin{cases}
\pa_t v^\ep - \De v^\ep + (\cJ(v^\ep)\cdot\na) (v^\ep\Phi_\ep) +\na p^\ep =0\\
\div v^\ep =0 \\
v^\ep|_{t=0} = v_0 
\end{cases}
}
in $\R^3 \times (0,T)$.

Since $v_0\in L^2_\uloc$ has no decay, it cannot be approximated by $L^2$-functions, as was done in \cite{KS} when $v_0 \in E^2$. Hence
the approximation solution $v^\ep$ cannot be constructed in the energy class $L^\infty(0,T; L^2(\R^3)) \cap L^2(0,T; \dot H^1(\R^3))$, and has to be constructed in $\cE_T$ directly.

Compared to \cite{LR,KS}, our mollified nonlinearity has an additional localization factor $\Phi_\ep$. It makes the decay of the Duhamel term apparent when the approximation solutions have no decay.

We first construct a mild solution $v^\ep$ of \eqref{reg.NS} in $\cE_T$. %

\begin{lemma}\label{mild.sol}
For each $0<\ep<1$ and $v_0$ with $\norm{v_0}_{L^2_\uloc}\leq B$, if $0<T< \min(1,c\ep^3B^{-2})$, we can find a unique solution $v=v^\ep$ to the integral form of \eqref{reg.NS}
\EQ{\label{op.rNS}
v(t) = e^{t\De}v_0 - \int_0^t e^{(t-s)\De} \mathbb{P} \na \cdot (\cJ(v) \otimes v \Phi_\ep)(s) ds
}
satisfying
\[
\norm{v}_{\cE_T} \leq 2C_0B,
\]
where $c>0$ and $C_0>1$ are absolute constants and $(a\otimes b)_{jk} = a_jb_k$. 
\end{lemma}

\begin{proof} %
Let $\Psi (v)$ be the map defined by the right side of \eqref{op.rNS} for $v \in \cE_T$. By  Lemma \ref{lemma24} and $T\le 1$,
\EQN{
\norm{\Psi (v)}_{\cE_T} 
&\lec \norm{v_0}_{L^2_\uloc}+ \norm{\cJ(v)\otimes v\Phi_\ep}_{U^{2,2}_T}
\\
&\lec \norm{v_0}_{L^2_\uloc}+ \norm{\cJ(v)}_{L^\infty(0,T;L^\infty(\R^3))}\norm{v}_{U^{2,2}_T}
\\
&\lec \norm{v_0}_{L^2_\uloc}+ \ep^{-\frac 32} \sqrt{T}\norm{v}_{U^{\I,2}_T}^2.
}
Thus
\[
\norm{\Psi (v)}_{\cE_T} \leq C_0\norm{v_0}_{L^2_\uloc} + C_1\ep^{-\frac 32}\sqrt{T} \norm{v}_{\cE_T}^2,
\]
for some constants $C_0,C_1>0$.
Similarly, for $v,u \in \cE_T$,
\[
\norm{\Psi (v)-\Psi (u)}_{\cE_T} \le 
C_1\ep^{-\frac 32}\sqrt{T} \bke{\norm{v}_{\cE_T}+ \norm{u}_{\cE_T}}
\norm{v-u}_{\cE_T}.
\]

By the Picard contraction theorem, if $T$ satisfies
\[
T< \frac {\ep^3}{64(C_0C_1B)^2} = c\ep^3B^{-2},
\]
then we can always find a unique fixed point $v\in \cE_T$ of $v = \Psi(v)$, i.e., \eqref{op.rNS},
satisfying
\[
\norm{v}_{\cE_T} \leq 2C_0B.\qedhere
\]
\end{proof}

\begin{lemma}\label{sol.rNS} Let $v_0\in L^2_\uloc$ with $\div v_0 =0$. 
For each $\ep\in (0,1)$,
we can find $v^\ep$ in $\cE_T$ and $p^\ep$ in $L^\infty(0,T;L^2(\R^3))$ for some positive $T=T(\ep, \norm{v_0}_{L^2_\uloc})$ which solves the localized-mollified Navier-Stokes equations \eqref{reg.NS} in the sense of distributions, and $\lim _{t \to 0_+} \norm{v^\ep(t)-v_0}_{L^2(E)}=0$ for any compact subset $E$ of $\R^3$.
\end{lemma}

\begin{proof} By Lemma \ref{mild.sol}, there is a mild solution $v^\ep \in \cE_T$ of \eqref{op.rNS} for some $T=T(\ep, \norm{v_0}_{L^2_\uloc})$. Apparently,  
\EQN{
\norm{v^\ep-e^{t\De}v_0}_{U^{\infty,2}_t} 
&= \norm{\int_0^t e^{(t-s)\De} \mathbb{P} \na \cdot (\cJ(v) \otimes v \Phi_\ep)(s) ds}_{U^{\infty,2}_t} \\
&\lec \norm{\cJ(v) \otimes v \Phi_\ep}_{U^{2,2}_t} 
\lec \ep^{-\frac 32}\sqrt{t}\norm{v}_{U^{\infty,2}_T}^2.
}
Also, for any compact subset $E$ of $\R^3$, we have $\norm{e^{t\De}v_0-v_0}_{L^2(E)}\to 0$ as $t$ goes to $0$; by Lebesgue's convergence theorem
\[
\norm{e^{t\De}v_0-v_0}_{L^2(E)} \leq \frac 1{(4\pi)^\frac32}\int e^{-\frac{|z|^2}{4}} \norm{v_0(\cdot -\sqrt{t}z) - v_0}_{_{L^2(E)}} dz \to 0, 
\]
as $t \to 0+$. Then, it follows that
$\lim _{t \to 0_+} \norm{v^\ep(t)-v_0}_{L^2(E)}=0$ for any compact subset $E$ of $\R^3$.

Note that $ e^{t\De}v_0$ with $v_0\in L^2_\uloc$ solves the heat equation in the distributional sense. Also, using $\div v_0 =0$, we can easily see that $\div e^{t\De}v_0 =0$.

On the other hand, $\cJ(v^\ep)\in L^\infty(\R^3\times [0,T])$ and $v^\ep\in \cE_{T}$ imply 
\[
\cJ(v^\ep) \otimes v^\ep \Phi_\ep\in L^\infty(0,T;L^2(\R^3))
\]
 and hence by the classical theory, $w^\ep = v^\ep-V$ and $p^\ep$ defined by
\EQ{\label{pep.def}
p^\ep = (-\De)^{-1}\pa_i\pa_j (\cJ(v_i^\ep) v_j^\ep \Phi_\ep) \in L^\infty(0,T;L^2(\R^3)).
}
 solves Stokes system with the source term $\na \cdot (\cJ(v^\ep) \otimes v^\ep \Phi_\ep)$ in the distribution sense.  

By adding the heat equation for $V$ with $\div V =0$ and the Stokes system for $(w^\ep,p^\ep)$, $v^\ep = V+w^\ep $ satisfies
\[
\pa_t v^\ep - \De v^\ep + (\cJ(v^\ep)\cdot \na) (v^\ep \Phi_\ep) + \na p^\ep =0
\]
in the sense of distribution. 
\end{proof}

To extract a limit solution from the family $(v^\ep, p^\ep)$ of approximation solutions, we need a uniform bound of $(v^\ep, p^\ep)$ on a uniform time interval $[0,T]$, $T>0$.  

\begin{lemma}\label{uni.est.ep} 
For each $\ep \in (0,1)$, let  $(v^\ep, p^\ep)$ be the solution on $\R^3 \times [0,T_\ep]$, for some $T_\ep>0$, to the localized-mollified Navier-Stokes equations \eqref{reg.NS} constructed in Lemma \ref{sol.rNS}. There is a small constant  $\e_1>0$, independent of $\e$ and $\norm{v_0}_{L^2_\uloc}^2$, such that, if  $T_\ep \leq T_0= {\e_1}(1+\norm{v_0}_{L^2_\uloc}^4)^{-1}$, then $v^\ep$ is uniformly bounded
\EQ{\label{uni.est.ep.li}
\norm{v^\ep}_{\cE_{T_\ep}} \leq C\norm{v_0}_{L^2_\uloc}, 
}
where the constant $C$ on the right hand side is independent of $\ep$ and $T_\ep$. 
\end{lemma}

\begin{proof}
Let $\phi_{x_0} = \Phi(\cdot-x_0)$ be a smooth cut-off function supported around $x_0$. For the convenience, we drop the index $x_0$. Starting from $v^\ep \in \cE_{T_\ep}$ and $p^\ep \in L^\I_{T_\ep} L^2$, and using the interior regularity theory for perturbed Stokes system with smooth coefficients, we have
\[
\norm{v^\ep, \pa_t v^\ep, \na v^\ep, \De v^\ep }_{L^\infty((\de,T_\ep)\times \R^3)}  <+\infty
\]
for any $\de\in (0,T_\ep)$. Using $2v^\ep \psi$ with $\psi \in C^\infty_c((0,T_\ep)\times \R^3)$ as a test function in \eqref{reg.NS}, we get
\EQN{%
2\int_0^T\!\! \int |\na v^\ep|^2\psi dxds
=& %
\int_0^T\!\!\int |v^\ep|^2(\pa_s \psi +\De \psi) dxds 
+\int_0^T\!\!\int |v^\ep|^2\Phi_\ep ({\cal J}_\ep(v^\ep) \cdot \na)\psi dxds\\
&+ 2\int_0^T\!\! \int p^\ep v^\ep \cdot \na \psi dxds
-\int_0^T\!\!\int |v^\ep|^2\psi ({\cal J}_\ep(v^\ep)\cdot \na)\Phi_\ep dxds.
}
Using $\lim_{t\to 0_+}\norm{v^\ep(t)-v_0}_{L^2(B_n)} =0$ for any $n\in\NN$ (Lemma \ref{sol.rNS}), we can show
\EQ{\label{LEI.vep00}
\int |v^\ep|^2 & \psi(x,t) dx+ 2\int_0^t\!\! \int  |\na v^\ep|^2\psi dxds
= \int |v_0|^2 \psi(\cdot,0) dx
\\
&+\int_0^t\!\!\int |v^\ep|^2(\pa_s \psi +\De \psi) dxds 
+\int_0^t\!\!\int |v^\ep|^2\Phi_\ep ({\cal J}_\ep(v^\ep) \cdot \na)\psi dxds\\
&+ 2\int_0^t\!\! \int p^\ep v^\ep \cdot \na \psi dxds
-\int_0^t\!\!\int |v^\ep|^2\psi ({\cal J}_\ep(v^\ep)\cdot \na)\Phi_\ep dxds
}
for any  $\psi \in C^\infty_c([0,T_\ep)\times  \R^3)$ and $0<t<T_\ep$.

We suppress the index $\ep$ in $v^\ep$ and $p^\ep$, and take $\psi(x,s)= \phi(x) \th(s)$ where $\th(s) \in C^\infty_c([0,T_\ep))$ and $\th(s)=1$ on $[0,t]$ to get
\EQ{\label{pre.lei.locex}
\norm{v(t) \phi}_2^2 &+ 2\norm{|\na v|\phi}_{L^2([0,t]\times \R^3)}^2\\ 
\lec& \norm{v_0}_{L^2_\uloc}^2 
+ \left|\int_0^t\!\int |v|^2|\De \phi^2|dxds \right|
+\left|\int_0^t\!\int |v|^2\phi^2 (\cJ(v)\cdot \na)\Phi_\ep dxds\right|\\
&+\left|\int_0^t\!\int |v|^2\Phi_\ep (\cJ(v)\cdot \na)\phi^2 dxds\right|
+\left|\int_0^t\!\int 2\hp (v\cdot \na)\phi^2  dxds\right|  \\
=& \norm{v_0}_{L^2_\uloc}^2 + I_1 +I_2+ I_3 +I_4,
}
where $\hp = \hp^\ep_{x_0}$ will be defined later in \eqref{phat.def} as a function satisfying $\na(p - \hp)=0 $ on $B(x_0, \frac 32)\times (0,T)$.   

The bounds of $I_1$, $I_2$ and $I_3$ can be easily obtained by H\"{o}lder inequalities,
\EQ{\label{I.123}
I_1 \lec \norm{v}_{U^{2,2}_t}^2, \quad \text{and} \quad 
I_2, I_3 \lec \norm{v}_{U^{3,3}_t}^3.
}
Here we have used $|\na \Phi_\ep|\lec \ep \leq 1$. 

On the other hand, $I_4$ can be estimated as
\[
I_4 \lec \norm{\hp}_{L^\frac32([0,t]\times B(x_0,\frac 32))}\norm{v}_{U^{3,3}_t}.
\]

Now, we define $\hp^\ep$ on $B(x_0,\frac 32)\times [0,T]$ by
\EQ{\label{phat.def}
\hp^\ep(x,t)
=&-\frac 13 \cJ(v^\ep)\cdot v^\ep\Phi_\ep(x,t) 
+ \pv \int_{B(x_0,2)}  K_{ij}(x-y) \cJ(v^\ep_i) v^\ep_j(y,t) \Phi_\ep(y) dy\\
&+  \int_{B(x_0,2)^c}  (K_{ij}(x-y)-K_{ij}(x_0-y)) \cJ(v^\ep_i) v^\ep_j(y,t) \Phi_\ep(y) dy \\
=& \ \hp^1 + \hp^2 + \hp^3.
}
Comparing the above with \eqref{pep.def} for $p^\ep$, which has the singular integral form 
\EQN{
p^\ep(x,t) = -\frac 13 \cJ(v^\ep)\cdot v^\ep(x,t)\Phi_\ep(x) + \pv \int  K_{ij}(x-y) \cJ(v_i^\ep) v_j^\ep(y,t) \Phi_\ep(y) dy,
}
we see that $p-\hp$ depends only on $t$, and hence $\na \hp =\na p$  on $B(x_0,\frac 32)\times [0,T]$.

Then, we take $L^\frac32([0,t]\times B(x_0,\frac 32))$-norm for each term to get
\[
\norm{\hp^1}_{L^\frac32([0,t]\times B(x_0,\frac 32))}
\lec \norm{v}_{U^{3,3}_t}^2,
\]
and 
\[
\norm{\hp^2}_{L^\frac32([0,t]\times B(x_0,\frac 32))}
\leq \norm{\hp^2}_{L^{\frac 32}([0,t]\times \R^3)}
\lec \norm{\cJ(v_i) v_j \Phi_\ep}_{L^\frac32([0,t]\times B(x_0,2))}
\lec \norm{v}_{U^{3,3}_t}^2. 
\]
The second inequality for $\hp^2$ follows from Calderon-Zygmund theorem. Finally, using 
\[
|K_{ij}(x-y)-K_{ij}(x_0-y)| \lec \frac {|x-x_0|}{|x_0-y|^4} 
\]
for $x\in B(x_0,\frac 32)$ and $y\in B(x_0,2)^c$, we have
\EQN{
\norm{\hp^3}_{L^\frac32([0,t]\times B(x_0,\frac 32))}
&\lec \norm{\int_{B(x_0,2)^c} \frac 1{|x_0-y|^4}\cJ(v_i) v_j(y,s) \Phi_\ep(y)dy}_{L^\frac32(0,t)}\\
&\lec \norm{\sum_{k=1}^\infty \frac {1}{2^{4k}} \int_{B(x_0,2^{k+1})}  |\cJ(v_i) v_j| (y,s) dy}_{L^\frac32(0,t)}\\
&\le \sum_{k=1}^\infty  \frac {1}{2^{4k}} \norm{\sum_{j=1}^{J_k} \int_{B(x^k_j,1)} |\cJ(v_i) v_j| (y,s) dy}_{L^\frac32(0,t)}\\
&\lec \sum_{k=1}^\infty \frac {J_k}{2^{4k}}\norm{\cJ(v_i) v_j}_{U^{\frac 32,\frac 32}_t} 
\lec \norm{v}_{U^{3,3}_t}^2.
}
Above we have taken $B(x_0,2^{k+1}) \subset \cup_{j=1}^{J_k} B(x^k_j,1)$ with $J_k \lec 2^{3k}$.

Therefore, we get
\EQ{\label{est.hp}
\norm{\hp}_{L^\frac32([0,t]\times B(x_0,\frac 32))} \lec \norm{v}_{U^{3,3}_t}^2
}
and 
\[
I_4 \lec \norm{v}_{U^{3,3}_t}^3. 
\]

Combining this with \eqref{I.123} and taking supremum on \eqref{pre.lei.locex} over $\{x_0\in \R^3\}$, we have
\[
\norm{v(t)}_{L^2_\uloc}^2 + 2\norm{\na v}_{U^{2,2}_t}^2 
\lec \norm{v_0}_{L^2_\uloc}^2 + \int_0^t \norm{v(s)}_{L^2_\uloc}^2 ds
+ \norm{v}_{U^{3.3}_t}^3.
\]
Then, using the interpolation inequality and Young's inequality,
\EQN{%
\norm{v}_{U^{3,3}_t}^3
&\lec \norm{v}_{U^{6,2}_t}^{3/2} \norm{v}_{U^{2,6}_t}^{3/2}
\\
& \lec \norm{v}_{L^6(0,t; L^2_\uloc )}^6 + \norm{v}_{L^2(0,t; L^2_\uloc )}^2 +\norm{\na v}_{U^{2,2}_t}^2,
}
we get
\EQ{\label{pre.uni.bdd}
\norm{v(t)}_{L^2_\uloc}^2 + \norm{\na v}_{U^{2,2}_t}^2 
\lec  \norm{v_0}_{L^2_\uloc}^2 + \int_0^t \norm{v(s)}_{L^2_\uloc}^2 ds
+ \int_0^t \norm{v(s)}_{L^2_\uloc}^6 ds.
}

Finally, we apply the Gr\"{o}nwall inequality, so that there is a small $\e_1>0$ such that, if $v^\ep$ exists on $[0,T]$ for
$T\le T_0$, $T_0= {\e_1}\bke{1+\norm{v_0}_{L^2_\uloc}^4}^{-1}$,
then we have
\[
\sup_{0<t<T} \norm{v^\ep(t)}_{L^2_\uloc} \lec \norm{v_0}_{L^2_\uloc}\left(1- \frac{Ct \norm{v_0}_{L^2_\uloc}^4}{\min(1,\norm{v_0}_{L^2_\uloc})^4}\right)^{-\frac 14}
\le \norm{v_0}_{L^2_\uloc}(1- {C\e_1} )^{-\frac 14}.
\]
Together with \eqref{pre.uni.bdd}, this completes the proof.
\end{proof}

\begin{lemma} \label{vep.uniform.interval}
The distributional solutions $\{(v^\ep,p^\ep)\}_{0<\ep<1}$ of \eqref{reg.NS} constructed in Lemma \ref{sol.rNS} can be extended 
to the uniform time interval $[0,T_0]$, where $T_0$ is as in Lemma \ref{uni.est.ep}.
\end{lemma}

\begin{proof}
We will prove it by iteration. For the convenience, we fix $0<\ep<1$ and drop the index $\ep$ in $v^\ep$ and $p^\ep$. Denote the uniform bound in Lemma \ref{uni.est.ep} by
\[
B=C(\norm{v_0}_{L^2_\uloc}), \quad B\geq \norm{v_0}_{L^2_\uloc}.
\]
If an initial data $v(t_0)$ satisfies $\norm{v(\cdot,t_0)}_{L^2_\uloc}\leq B$, by Lemma \ref{sol.rNS}, we get $S=S(\ep,B)>0$ and a unique solution $v(x,t+t_0)$ on $\R^3\times [0,S]$ to \eqref{op.rNS} satisfying 
\[
\norm{v(t+t_0)}_{\cE_{S}} \leq 2C_0 B.
\]

Now, we start the iteration scheme. Since $\norm{v_0}_{L^2_\uloc}\leq B$, a unique solution $v$ exists in $ \cE_{S}$ to \eqref{op.rNS}. By Lemma \ref{sol.rNS} and Lemma \ref{uni.est.ep}, $v$ satisfies
\[
\norm{v}_{\cE_{S}} \leq B.
\]
Then, we choose $\tau \in (\frac 34S,S)$, so that $\norm{v(\tau)}_{L^2_\uloc}\leq B$, and hence we obtain a solution $\td {v} \in \cE(\tau,\tau+S)$ to
\[
\td {v}(t) = e^{(t-\tau)\De}v|_{t=\tau} + \int_{\tau}^{t} e^{(t-s)\De}\mathbb{P}\na \cdot N^\ep(\td v) (s) ds,
\]
where we denote $N^\ep(v) = \cJ(v) \otimes v \Phi_\ep$.

Denote the glued solution by $u(x,t) = v(x,t) 1_{[0,\tau]}(t) + \td{v}(x,t) 1_{(\tau,\tau+S]}(t)$, where $1_E$ is a characteristic function of a set $E\subset [0,\infty)$. We claim that it solves \eqref{op.rNS} in $(0,\tau+S)$; it is obvious for $t\in (0,\tau]$, and for $t\in (\tau,\tau+S]$,
\EQN{
u(t) =& \ \td v(t)
\\
=& e^{(t-\tau)\De}\left(e^{\tau\De}v_0 + \int_0^{\tau} e^{(\tau-s)\De}\mathbb{P}\na \cdot N^\ep(v)(s) ds\right)
+ \int_{\tau}^{t} e^{(t-s)\De}\mathbb{P}\na \cdot N^\ep(\td v)(s) ds\\
=& \ e^{t\De} v_0 + \int_0^{\tau} e^{(t-s)\De}\mathbb{P}\na \cdot N^\ep(v)(s) ds
+ \int_{\tau}^{t} e^{(t-s)\De}\mathbb{P}\na \cdot N^\ep(\td v)(s) ds\\
=& \ e^{t\De} v_0 + \int_0^t e^{(t-s)\De}\mathbb{P}\na \cdot N^\ep(u)(s) ds.
}
By Lemma \ref{uni.est.ep} again, it satisfies
\[
\norm{u}_{\cE(0,\tau+S)} \leq B.
\]
By uniqueness, we get $u= v$ for $0\le t \le S$. In other words, $u$ is an extension of $v$. 

Repeat this until the extended solution exists on $[0,T_0]$. Since at each iteration, we can extend the time interval by at least $\frac 34 S$, in finite numbers of iterations, we have a distributional solution $(v^\ep,p^\ep)$ of \eqref{reg.NS} on $\R^3\times [0,T_0]$.  
\end{proof}

\begin{proof}[Proof of Theorem \ref{loc.ex}] 
For $0< \ep \ll 1$, let $(v^\ep,\bar{p}^\ep)$ be the distributional solution to the localized-mollified Navier-Stokes equations \eqref{reg.NS} on $\R^3 \times [0,T]$ constructed in Lemmas \ref{sol.rNS} and \ref{vep.uniform.interval}, where $T=T(\norm{v_0}_{L^2_\uloc})$ is independent of $\ep$. By Lemma \ref{uni.est.ep},
\[
\norm{v^\ep}_{\cE_T}\le C(\norm{v_0}_{L^2_\uloc}).
\]
We then define $p^\ep \in L^{\frac 32}_\loc( [0,T]\times \R^3)$ by
\EQ{\label{p.ep}
p^\ep (x,t)
&=-\frac 13 \cJ(v^\ep)\cdot v^\ep(x,t)\Phi_\ep(x) 
+ \pv \int_{B_2}  K_{ij}(x-y) N^\ep_{ij}(y,t) dy\\
&\quad + \pv \int_{B_2^c}  (K_{ij}(x-y)-K_{ij}(-y)) N^\ep_{ij}(y,t) dy,
\\
N^\ep_{ij}(y,t)  & = \cJ(v_i^\ep) v_j^\ep(y,t) \Phi_\ep(y).
}
Because $N^\ep_{ij} \in L^\infty(0,T; L^2(\R^3))$, the right side of \eqref{p.ep} is defined in $L^\infty(0,T; L^2(\R^3)) + L^\infty(0,T)$.
Note that $\na(\bar{p}^\ep -p^\ep) =0$ because
\[
(\bar{p}^\ep -p^\ep)(t) =  \int_{B_2^c}  K_{ij}(-y) \cJ(v_i^\ep) v_j^\ep(y,t) \Phi_\ep(y) dy \in L^{\frac 32}(0,T).
\]
Therefore, $(v^\ep, p^\ep)$ is another distributional solution to the localized-mollified equations \eqref{reg.NS}. We will show that for each $n\in \NN$, $p^\ep$ has a bound independent of $\ep$ in $L^\frac 32([0,T]\times B_{2^n})$. We drop the index $\ep$ in $v^\ep$ and $p^\ep$ for a moment.

For $n\in \NN$, 
we rewrite \eqref{p.ep} for $x\in B_{2^{n}}$ as follows.
\EQN{
p(x,t) =&-\frac 13 \cJ(v)\cdot v(x,t)\Phi_\ep(x) 
+ \pv \int_{B_2}  K_{ij}(x-y) N_{ij}^{\ep}(y,t) dy\\
&+  \left(\pv \int_{B_{2^{n+1}}\setminus B_2}+\pv \int_{B_{2^{n+1}}^c}\right)  (K_{ij}(x-y)-K_{ij}(-y)) N_{ij}^{\ep}(y,t) dy\\
=& \ p_1 + p_2 + p_3 +p_4.
} 
All $p_i$ are defined in $L^\infty(0,T; L^2)+L^\infty(0,T)$.

By Lemma \ref{uni.est.ep}, we have
\EQ{\label{est.Nij}
\norm{N_{ij}^\ep}_{U^{\frac 32,\frac32}_T} \lec \norm{\cJ(v)}_{U^{3,3}_T}\norm{v}_{U^{3,3}_T}
\leq C(\norm{v_0}_{L^2_\uloc}),
}
and
\EQ{\label{est.Nij2}
\norm{{N}_{ij}^{\ep}}_{L^{\frac 32}([0,T]\times B_{2^n})} 
\lec 2^{2n}\norm{\cJ(v)}_{U^{3,3}_T}\norm{v}_{U^{3,3}_T}
\leq C(n, \norm{v_0}_{L^2_\uloc}), \quad\forall n\in \NN.
}

Then, the bound of $p_1$ can be obtained since
\[
\norm{p_1}_{L^{\frac 32}([0,T]\times B_{2^n})}\lec  \sum_{i=1}^3\norm{{N}_{ii}^{\ep}}_{L^{\frac 32}([0,T]\times B_{2^n})}.
\]
Using Calderon-Zygmund theorem, we get
\[
\norm{p_2}_{L^{\frac 32}([0,T]\times B_{2^n})}
\lec \norm{{N}_{ij}^{\ep}}_{L^{\frac 32}([0,T]\times B_{2})},
\]
and 
\[
\norm{p_{31}}_{L^{\frac 32}([0,T]\times B_{2^n})}
\lec \norm{{N}_{ij}^{\ep}}_{L^{\frac 32}([0,T]\times B_{2^{n+1}})}, 
\]
where 
\[
p_{31}(x,t) = \pv \int_{B_{2^{n+1}}\setminus B_2} K_{ij}(x-y) \cJ(v_i)v_j(y,t)\Phi_\ep(y)dy.
\]
On the other hand, $p_{32}= p_3-p_{31}$ satisfies
\EQN{
\norm{p_{32}}_{L^{\frac 32}([0,T]\times B_{2^n})}
&\lec 2^{2n} \norm{\frac 1{|y|^3}}_{L^3(B_{2^{n+1}}\setminus B_2)} 
\norm{{N}_{ij}^{\ep}}_{L^{\frac 32}([0,T]\times B_{2^{n+1}})}
\\
&\lec 2^{2n} \norm{{N}_{ij}^{\ep}}_{L^{\frac 32}([0,T]\times B_{2^{n+1}})}.
}

Since for $x\in B_{2^{n}}$ and $y\in B_{2^{n+1}}^c$, we have
\[
|K_{ij}(x-y)-K_{ij}(-y)|\lec \frac{|x|}{|y|^4} \lec \frac{2^n}{|y|^4},
\]
the bound of $p_4$ can be obtained as
\EQN{
\norm{p_4}_{L^\frac 32([0,T]\times B_{2^n})}
&\lec 2^{2n} \norm{p_4}_{L^\frac 32(0,T; L^\infty( B_{2^n}))}
\lec 2^{3n}\norm{\int_{B_{2^{n+1}}^c} \frac 1{|y|^4} |N_{ij}^\ep|(y,t) dy}_{L^\frac 32(0,T)}\\
&\lec 2^{3n} \sum_{k=n+1}^\infty \frac 1{2^{4k}}\norm{N_{ij}^\ep}_{L^\frac 32(0,T;L^1(B_{2^{k+1}}))}
\lec_n  \norm{N_{ij}^\ep}_{U^{\frac 32,1}_T}.
}

Adding the estimates and using \eqref{est.Nij}-\eqref{est.Nij2}, we get for each $n\in \NN$, 
\EQ{\label{uni.bdd.p}
\norm{p^\ep}_{L^\frac 32([0,T]\times B_{2^n})} \leq C(n,\norm{v_0}_{L^2_\uloc}).
}

Now, we find a limit solution of $(v^\ep, p^\ep)$ up to subsequence on each $[0,T]\times B_{2^n}$, $n\in \NN$. 

First, construct the solution $v$ on the compact set $[0,T]\times B_2$. By uniform bounds on $v^\ep$ and the compactness argument, we can extract a sequence $v^{1,k}$ from $\{v^\ep\}$ satisfying
\EQN{
&v^{1,k} \stackrel{\ast}{\rightharpoonup} v^1 \qquad\qquad \text{in } L^\infty(0,T;L^2(B_2)),\\
&v^{1,k} \rightharpoonup v^1 \qquad\qquad\text{in }L^2(0,T;H^1(B_2)),\\
&v^{1,k} \rightarrow v^1 \qquad\qquad\text{in }L^3(0,T;L^3(B_2)),\\ 
&{\cal J}_{1,k}(v^{1,k}) \rightarrow v^1 \quad\text{ in }L^3(0,T;L^3(B_{2^-})),
}
as $k\to \infty$. Let $v=v^1$ on $[0,T]\times B_2$. 

Then, we extend $v$ to $[0,T]\times B_4$ as follows. In a similar way to getting $v^1$, we can find a subsequence $\{(v^{2,k}, p^{2,k})\}_{k\in \NN}$ of $\{(v^{1,k}, p^{1,k})\}_{k\in \NN}$ which satisfies the following convergence:
\EQN{
&v^{2,k} \stackrel{\ast}{\rightharpoonup} v^2 \qquad\qquad \text{in } L^\infty(0,T;L^2(B_4)),\\
&v^{2,k} \rightharpoonup v^2 \qquad\qquad\text{in }L^2(0,T;H^1(B_4)),\\
&v^{2,k} \rightarrow v^2 \qquad\qquad\text{in }L^3(0,T;L^3(B_4)),\\ 
&{\cal J}_{2,k}(v^{2,k}) \rightarrow v^2\quad\text{    in }L^3(0,T;L^3(B_{4^-})),
}
as $k\to \infty$. Here, we can easily check that $v^2=v^1$ on $[0,T]\times B_2$, so that $v=v^2$ is the desired extension. By repeating this argument, we can construct a sequence $\{v^{n,k}\}$ and its limit $v$. Indeed, by the diagonal argument, $v$ can be approximated by 
\[
v^{(k)}= \begin{cases}
v^{k,k} & [0,T]\times B_{2^k},\\
0& \text{otherwise}
\end{cases}, \quad \forall k\in \NN
\]
More precisely, on each $[0,T]\times B_{2^n}$, $\{v^{(k)}\}_{k=n}^\infty$ enjoys the same convergence properties as above. This follows from that $\{v^{m,j}\}_{j\in \NN}$, $m\geq n$ is a subsequence of $\{v^{n,j}\}_{j\in \NN}$. Indeed, for each $v^{k,k}$, $k\geq n$, we can find $j_k\geq k$ such that 
\[
v^{k,k}= v^{n,j_k}. 
\]
Then, by its construction, for each $n\in \NN$, $\{v^{(k)}\}_{k=n}^\infty$ satisfies 
\begin{align}
&v^{(k)} \stackrel{\ast}{\rightharpoonup} v \qquad\qquad \text{in } L^\infty(0,T;L^2(B_{2^n})), \label{conv.weak.star} \\
&v^{(k)} \rightharpoonup v \qquad\qquad\text{in }L^2(0,T;H^1(B_{2^n})),\label{conv.weak}\\
&v^{(k)} \rightarrow v \qquad\qquad\text{in }L^3(0,T;L^3(B_{2^n})), \label{conv.strong}\\ 
&{\cal J}_{(k)}(v^{(k)}) \rightarrow v \quad\text{    in }L^3(0,T;L^3(B_{{2^n}^-})) \label{conv.moli}
\end{align}
as $k\to \infty$. 
Furthermore, since $v^\ep$ are uniformly bounded in $\cE_T$, we can easily see that $v\in \cE_T$ and $v\in U^{3,3}_T$,
\[
\norm{v}_{\cE_T} + \norm{v}_{U^{3,3}_T} \le C(\norm{v_0}_{L^2_\uloc}).
\]

Now, we construct a pressure $p$ corresponding to $v$. Using \eqref{p.ep}, we define $p^{(k)}$ by  
\EQ{\label{pk.def}
p^{(k)}(x,t)
=&-\frac 13 {\cal J}_{(k)}(v^{(k)})\cdot v^{(k)}(x,t)\Phi_{(k)}(x) \\
&+ \pv \int_{B_2}  K_{ij}(x-y) {\cal J}_{(k)}(v_i^{(k)}) v_j^{(k)}(y,t) \Phi_{(k)}(y) dy\\
&+ \pv \int_{B_2^c}  (K_{ij}(x-y)-K_{ij}(-y)) {\cal J}_{(k)}(v^{(k)}_i) v^{(k)}_j(y,t) \Phi_{(k)}(y) dy.
}
where $\Phi_{(k)} = \Phi_{\ep_k}$ for $\ep_k$ satisfying $v^{k,k} = v^{\ep_k}$. Also define
\EQ{\label{p.def}
p(x,t)  = \lim_{n \to \infty} \bar p^n(x,t)
}
where $\bar p^n(x,t)$ is defined for $|x|<2^{n}$ by
\EQ{\label{p34.dec}
\bar p^n(x,t) 
=&-\frac 13 |v(x,t)|^2 
+\pv \int_{B_2}  K_{ij}(x-y) v_i v_j(y,t) dy + \bar p^n_3+\bar p^n_4,
}
with
\EQN{
\bar p^n_3(x,t) &=\pv \int_{B_{2^{n+1}}\setminus B_2}  (K_{ij}(x-y)-K_{ij}(-y)) v_iv_j(y,t)\,  dy ,
\\
\bar p^n_4(x,t) &=\int_{B_{2^{n+1}}^c}  (K_{ij}(x-y)-K_{ij}(-y)) v_iv_j(y,t) \, dy .
}
The first two terms in $\bar p^n$ are defined in $U^{\frac32, \frac 32}_T$ and independent of $n$. Among the last two terms, $\bar p^n_4$ converges absolutely but $\bar p^n_3$ only in $U^{\frac32, \frac 32}_T$. By estimates similar to those for $p^\ep$, we get $\bar p^n_3, \bar p^n_4 \in L^{3/2}((0,T)\times B_{2^n})$ and
\[
\bar p^n_3+\bar p^n_4 = \bar p^{n+1}_3+\bar p^{n+1}_4 , \quad \text{in}\quad L^{3/2}((0,T)\times B_{2^n})
\]
Thus $\bar p^n(x,t)$ is independent of $n$ for $n> \log_2 |x|$.

Our goal is to show that the strong convergences \eqref{conv.strong}-\eqref{conv.moli} of $\{v^{(k)}\}$ gives
\EQ{\label{conv.p.st}
p^{(k)} \to {p} \quad \text{ in }L^{\frac 32}([0,T]\times B_{2^{n}}), \quad \text{ for each }n\in \NN, 
}
Let $N_{ij}^{(k)} = {\cal J}_{(k)}(v^{(k)}_i) v^{(k)}_j\Phi_{(k)}$ and $N_{ij} = v_iv_j$. For any fixed $R>0$, we have
\EQ{\label{conv.in.pres}
&\norm{N_{ij}^{(k)}-N_{ij}}_{L^{\frac 32}([0,T]\times B_R)}
\\
&\leq \norm{\left({\cal J}_{(k)}(v_i^{(k)})-v_i\right) v_j^{(k)}\Phi_{(k)}}_{L^{\frac 32}([0,T]\times B_R)}
+ \norm{v_i (v_j^{(k)} - v_j)\Phi_{(k)}}_{L^{\frac 32}([0,T]\times B_R)}\\
&\quad +\norm{v_i v_j(1-\Phi_{(k)})}_{L^{\frac 32}([0,T]\times B_R)}\\
&\lec \norm{{\cal J}_{(k)}(v^{(k)})-v}_{L^3([0,T]\times B_R)}\norm{v^{(k)}}_{L^3([0,T]\times B_R)}\\
&\quad+\norm{v^{(k)}-v}_{L^3([0,T]\times B_R)}\norm{v}_{L^3([0,T]\times B_R)}
+\norm{|v|^2(1-\Phi_{(k)})}_{L^{\frac 32}([0,T]\times B_R)}
\longrightarrow 0
}
by \eqref{conv.strong}, \eqref{conv.moli}, and Lebesgue dominated convergence theorem. Then, it provides the convergence of $p^{(k)}$ to $p$: On $[0,T]\times B_{2^{n}}$, for $m >n$,
\EQN{
p^{(k)}-p 
=&-\frac 13 \operatorname{tr}(N^{(k)}-N)
+ \pv \int_{B_2}  K_{ij}(\cdot-y) (N_{ij}^{(k)}-N_{ij})(y) dy\\
&+  \left[\pv \int_{B_{2^{n+1}}\setminus B_2}+\int_{B_{2^{m}}\setminus B_{2^{n+1}}} +\int_{B_{2^{m}}^c}\right]    (K_{ij}(\cdot-y)-K_{ij}(-y)) (N_{ij}^{(k)}-N_{ij})(y) dy\\
=& \ q_1 + q_2 + q_3 + q_4 +q_5.
}

In a similar way to getting \eqref{uni.bdd.p}, we have
\[
\norm{q_1,q_2, q_3}_{L^\frac 32([0,T]\times B_{2^{n}})}\lec_n \norm{N^{(k)}-N }_{L^\frac 32([0,T]\times B_{2^{n+1}})},
\]
and
\[
\norm{q_4}_{L^\frac 32([0,T]\times B_{2^{n}})}\lec \norm{N^{(k)}-N }_{L^\frac 32([0,T]\times B_{2^{m}})},
\]

On the other hand, using 
\[
|K_{ij}(x-y)-K_{ij}(-y)|\lec \frac {|x|}{|y|^4}
\]
we obtain%
\EQN{
\norm{q_5}_{L^\frac 32([0,T]\times B_{2^{n}})}
\lec \frac {2^{3n}}{2^m} \left(\norm{v}_{U^{3,3}_T}^2 +\norm{{\cal J}_{(k)} (v^{(k)}) }_{U^{3,3}_T}\norm{v^{(k)}}_{U^{3,3}_T}\right) \leq C(n,\norm{v_0}_{L^2_\uloc},T)\frac 1{2^m}.
}

Therefore, for fixed $n$, if we choose sufficiently large $m$, we can make $q_5$ very small in $L^\frac 32([0,T]\times B_{2^{n}})$ and then for sufficiently large $k$, $q_1$, $q_2$, $q_3$, and $q_4$ also become very small in $L^\frac 32([0,T]\times B_{2^{n}})$ because of \eqref{conv.in.pres}. This gives the desired convergence \eqref{conv.p.st} of $p^{(k)}$ to $p$.

Now, we check that $(v,p)$ is a local energy solution. It is easy to prove that $(v,p)$ solves the Navier-Stokes equation in distributional sense by using the distributional form %
of \eqref{reg.NS} for $(v^{(k)},p^{(k)})$ and the convergence \eqref{conv.weak.star}-\eqref{conv.moli} and \eqref{conv.p.st}-\eqref{conv.in.pres}. For example, 
for any $\xi \in C_c^\infty((0,T)\times \R^3;\R^3)$,
\[
\left.
\begin{split}
\int_0^T\int v^{(k)}\cdot \pd_t \xi dxdt &\to \int_0^T\int v\cdot \pd_t \xi dxdt
\\
\int_0^T\int {\cal J}_{(k)}(v^{(k)})(v^{(k)}\Phi_{(k)}) : \nb \xi dxdt
&\to \int_0^T\int v \otimes v: \nb \xi dxdt
\end{split}
\right\}
\quad \text{as}\quad k \to \infty.
\]
Since we have 
\EQN{
\int_0^t\!\int &(\De v- (v \cdot \na )v - \na p )\cdot \phi dxdt\\
\leq& \left| \int_0^t\! \int \na v \cdot \na \phi dx dt\right|
+ \left|\int_0^t\! \int v(v\cdot \na ) \phi dxdt\right|
+ \left|\int_0^t\! \int p \div \phi dxdt\right|\\
\lec& \norm{\na v}_{L^2(0,T;L^2(B_{2^n}))} \norm{\na \phi}_{L^2(0,T;L^2(\R^3))}\\
&+ \bke{\norm{v}_{L^3(0,T;L^3(B_{2^n}))}^2+ \norm{p}_{L^\frac 32(0,T;L^{\frac 32}(B_{2^n}))}} \norm{\na \phi}_{L^3(0,T;L^3(\R^3))}\\
\leq &\ C(n,T,\norm{v_0}_{L^2_\uloc})\norm{\na \phi}_{L^3(0,T;L^3(\R^3))},
}
for any $\phi\in C_c^\infty([0,T]\times B_{2^n})$, $n\in \NN$, it follows that
\EQN{
\pa_t v = \De v - (v \cdot \na )v - \na p  \in X_n
}
for any $n\in \NN$, where $X_n$ is the dual space of $L^3(0,T;{W}^{1,3}_0(B_{2^n}))$.

With this bound of $\pd_t v$, for each $n \in \NN$, we may redefine $v(t)$ on a measure-zero subset $\Si_n$ of $[0,T]$ such that the function 
\EQ{\label{weak.conti}
t \longmapsto \int_{\R^3} v(x,t)\cdot \zeta(x)\, dx
}
is continuous for any vector $\zeta\in C_c^\infty(B_{2^n})$. Redefine $v(t)$ recursively for all $n$ so that \eqref{weak.conti} is true for any $\zeta\in C_c^\infty(\R^3)$.  It is then true for any $\zeta\in L^2(\R^3)$  with a compact support using $v\in L^\infty(0,T;L^2_\uloc)$.

Furthermore, consider the local energy equality \eqref{LEI.vep00} for $(v^{(k)},p^{(k)})$ on the time interval $(0,T)$ for a non-negative $\psi\in C_c^\infty([0,T)\times \R^3)$. The first term $\int |v^{(k)}|^2 \psi(x,T) dx$ vanishes.
Taking limit infimum as $k$ goes to infinity, and using the weak convergence \eqref{conv.weak} and
the strong convergence \eqref{conv.strong}-\eqref{conv.moli} and \eqref{conv.p.st}-\eqref{conv.in.pres},
we get
\EQ{\label{LEI.v.nt}
2\int_0^T\!\!\int & |\na v|^2\psi\, dxds
\leq \int|v_0|^2\psi(\cdot,0) dx   \\
&+\int_0^T\!\!\int |v|^2(\pa_s\psi +\De\psi) +(|v|^2+2\hp) (v\cdot \na)\psi\, dxds,
}
for any non-negative $\psi\in C_c^\infty([0,T)\times \R^3)$. 

Then, for any $t\in (0,T)$ and non-negative $\ph\in C_c^\infty([0,T)\times \R^3)$, take  $\psi(x,s) = \ph(x,s) \th_\e(s)$, $\ep \ll 1$, where $\th_\e(s) =\th\left(\frac{s -t}{\ep}\right)$ for some $\th\in C^\infty(\R)$ such that $\th(s)=1$ for $s\le 0$ and $\th(s) = 0$ for $s\ge 1$, and $\th'(s)\le 0$ for all $s$. Note that $\th_\e(s) = 1$ for $s \le t$ and $\th_\e(s) =0$ for $s \ge t+\ep$. Sending $\e \to 0$ and using
\[
\int  |v(t)|^2 \ph\, dx \leq \liminf_{\ep\to 0} \int_0^t \int |v|^2\ph (-\th_\e') \,dx\,ds
\]
due to the weak local $L^2$-continuity \eqref{weak.conti},
we get 
\EQ{\label{pre.lei}
\int & |v(t)|^2 \ph\, dx + 2\int_0^t\! \int |\na v|^2\ph\, dxds\\
&\leq \int |v_0|^2 \ph(\cdot,0) dx + \int_0^t\! \int \bket{ |v|^2(\pd_s \ph+ \De \ph)  
+(|v|^2+2\hp) (v\cdot \na)\ph} dxds
}
for any $t\in (0,T)$ and non-negative $\ph\in C_c^\infty([0,T)\times \R^3)$. The local energy inequality \eqref{LEI} is a special case of 
\eqref{pre.lei} for test functions vanishing at $t=0$.

Sending $t\to 0_+$ in \eqref{pre.lei} we get $\limsup_{t \to 0_+} \int |v(t)|^2 \ph\, dx\le \int |v_0|^2 \ph(\cdot,0) dx$ for any non-negative $\ph\in C_c^\infty$. Together with the weak continuity \eqref{weak.conti}, we get  $\lim_{t \to 0_+} \int_{B_n} |v(x,t)-v_0(x)|^2 dx =0$ for any $n \in \NN$.

Finally, we consider the decomposition of the pressure. Recall that the pressure $p$ is defined recursively by \eqref{p.def}-\eqref{p34.dec}. For any $x_0\in \R^3$ define $\hp_{x_0}\in L^{\frac 32}([0,T]\times B(x_0, \frac 32))$ by \eqref{hp.def}, i.e.,
\EQN{
\hp_{x_0}(x,t)  =& -\frac 13 |v|^2 (x,t)
+\pv  \int_{B(x_0,2)}  K_{ij}(x-y) v_i v_j(y,t) dy\\
&+  \int_{B(x_0,2)^c}  (K_{ij}(x-y)-K_{ij}(x_0-y)) v_iv_j(y,t)  dy.
}
Let $c_{x_0} = p-\hp_{x_0}$. If $ B(x_0,\frac 32)\subset B_{2^n}$, then
\EQ{\label{cx0.def}
c_{x_0}(t)
=&  \int_{B_{2^{n+1}}\setminus B(x_0,2)}  K_{ij}(x_0-y) v_iv_j(y,t)  dy\\
&- \int_{B_{2^{n+1}}\setminus B_2}  K_{ij}(-y) v_iv_j(y,t)  dy\\
&+ \int_{B_{2^{n+1}}^c}  (K_{ij}(x_0-y)-K_{ij}(-y)) v_iv_j(y,t)  dy .%
}
Note that $c_{x_0} \in L^{3/2}(0,T)$, and $c_{x_0}(t)$ is independent of $x\in B(x_0,\frac 32)$, $n$, and $T$. Therefore, we get the desired decomposition \eqref{pressure.decomp} of the pressure. 
\end{proof}

\begin{remark}
Our approach in this section is similar to that in Kikuchi-Seregin \cite{KS}. However, there are two significant differences:
\begin{enumerate}
\item Since we include initial data $v_0$ not in $E^2$, we add an additional localization factor $\Phi_{(k)}$ to the nonlinearity in the localized-mollified equations \eqref{reg.NS}. Our approximation solutions $v^\ep$ live in $L^2_\uloc$ and are no longer in the usual energy class. 

\item The pressure $p$ and $c_{x_0}$ are implicit in \cite{KS}, but  are explicit in this paper. We first specify the formula \eqref{p.def} of the pressure and then justify the strong convergence and decomposition. In particular, our $c_{x_0}(t)$ is given by \eqref{cx0.def} and independent of $T$.
\end{enumerate}
\end{remark}

\begin{remark}
Estimate \eqref{est.hp} and its proof for $\hp^{\,\e}_{x_0}$ are not limited to our approximation solutions. They are in fact also valid for any local energy solution $(v,p)$ in $(0,T)$ with local pressure $\hp_{x_0}$ given by \eqref{hp.def}, that is,
\EQ{\label{est.hp2}
\norm{\hp_{x_0}}_{L^\frac32([0,t]\times B(x_0,\frac 32))} \le C \norm{v}_{U^{3,3}_t}^2, \quad \forall t<T,
}
with a constant $C$ independent of $t,T$.
\end{remark}

\section{Spatial decay estimates}\label{decay.est.sec}

Recall that our initial data $v_0\in E^2_{\si} + L^3_{\uloc,\si}$. In Sections \ref{decay.est.sec} and \ref{global.sec}, we decompose 
\EQ{\label{v0.dec}
v_0 = w_0+u_0, \quad w_0\in E^2_\si,  \quad u_0\in L^3_{\uloc,\si}.
}
Our goal in this section is to show that, although the solution $v$ has no spatial decay, its difference from the linear flow, $w= v-V$, $V(t)=e^{t\De}u_0$, does decay due to the decay of the oscillation of $u_0$. Here, the oscillation decay of $u_0$ follows from that of $v_0$ and $w_0\in E^2$. The main task is to show that the contribution from the nonlinear source term 
\[
(V \cdot \nb) V = \nb \cdot (V \otimes V)
\]
has decay, although $V$ itself does not. 
On the other hand, we also need the decay of the pressure. However, $\hp_{x_0}$ given by \eqref{hp.def} 
does not decay. Thus
we need a different decomposition of the pressure $p$ near each point $x_0\in \R^3$.

\begin{lemma}[New pressure decomposition]\label{new.decomp}
Let $v_0 = w_0+u_0$ with $ w_0\in E^2_\si$ and $u_0\in L^3_{\uloc,\si}$. Let $(v,p)$ be any local energy solution of \eqref{NS} with initial data $v_0$ in $\R^3 \times (0,T)$, $0<T<\infty$. 
Then, for each $x_0\in \R^3$, we can find $q_{x_0}\in L^\frac 32(0,T)$ such that 
\[
p(x,t) =\cp_{x_0}(x,t) +q_{x_0}(t) \quad \text{in }L^\frac 32((0,T)\times B(x_0, \tfrac32))
\]
where
\EQ{\label{cp}
\cp_{x_0}
=&-\frac 13 (|w|^2 + 2w\cdot V) 
+ \pv \int_{B(x_0,2)}  K_{ij}(\cdot-y) (w_iw_j + V_iw_j+ w_iV_j)(y) dy\\
&+ \int_{B(x_0,2)^c}  (K_{ij}(\cdot-y)-K_{ij}(x_0-y)) (w_iw_j + V_iw_j+ w_iV_j)(y) dy \\
&+ \int K_i(\cdot-y)[ (V\cdot \na)V_i]\rho_2 (y) dy \\
&+ \int (K_{ij}(\cdot -y)-K_{ij}(x_0-y)) V_iV_j(1-\rho_2)(y)dy \\
&+ \int (K_{i}(\cdot -y)-K_{i}(x_0-y)) V_iV_j(\pa_j\rho_2)(y)dy.
}
Here, $w=v-V$, $V(t)=e^{t\De}u_0$, $K_i = \pa_i K$, $K_{ij}=\pa_{ij}K$, $K(x)= \frac 1{4\pi|x|}$, and $\rho_2 = \Phi(\frac{\cdot-x_0}{2})$.
\end{lemma}

\begin{proof} Consider $(x,t)\in B(x_0, \frac 32)\times (0,T)$.
Let $F_{ij} = w_iw_j + V_iw_j+ w_iV_j$ and $G_{ij} = V_iV_j$. 
Substituting $v=V+w$ in \eqref{hp.def}, we get
\begin{align}
\hp_{x_0} &= p_{x_0}^F + p_{x_0}^G \nonumber\\
\begin{split} \label{pFx0}
p_{x_0}^F &= -\frac 13 \tr F 
+ \pv \int_{B(x_0,2)}  K_{ij}(\cdot-y) F_{ij}(y) dy
\\
&\quad + \int_{B(x_0,2)^c}  (K_{ij}(\cdot-y)-K_{ij}(x_0-y)) F_{ij}(y) dy
\end{split}
\end{align}
and
\EQN{
p_{x_0}^G &= -\frac 13 \tr G 
+ \pv \int_{B(x_0,2)}  K_{ij}(\cdot-y) G_{ij}(y) dy
\\
&\quad + \int_{B(x_0,2)^c}  (K_{ij}(\cdot-y)-K_{ij}(x_0-y)) G_{ij}[\rho_2+ (1-\rho_2)](y) dy
\\
 &= -\frac 13 \tr G 
+ \pv \int  K_{ij}(\cdot-y) G_{ij}\rho_2(y) dy + p_{x_0,\mathrm{far}}^G + \td q_{x_0}(t) ,
}
where
\EQN{
p_{x_0,\mathrm{far}}^G &= \int  (K_{ij}(\cdot-y)-K_{ij}(x_0-y)) G_{ij}(1-\rho_2)(y) dy,
\\
\td q_{x_0}(t) &=  -\int_{B(x_0,2)^c}  K_{ij}(x_0-y) G_{ij}\rho_2(y) dy.
}
Integrating by parts the principle value integral, we get
\EQN{
p_{x_0}^G &= \int%
K_{i}(\cdot-y)\pd_j[ G_{ij}\rho_2(y) ]dy
 + p_{x_0,\mathrm{far}}^G + \td q_{x_0}(t).
}
Note $\pd_j[ G_{ij}\rho_2] = (V\cdot \nb V_i)\rho_2 + G_{ij} \pd_j \rho_2$. Denote
\[
\widehat q_{x_0}(t)=\int K_{i}(x_0-y) V_iV_j(\pa_j\rho_2)(y)dy.
\]
We get
\EQ{\label{hp.cp}
\hp_{x_0}(x,t) &= p_{x_0}^F
+ \int K_i(\cdot-y) (V\cdot \na)V_i\rho_2 (y) dy + p^G_{x_0,\text{far}} + \td q_{x_0}(t)
\\
&\quad + \int (K_{i}(\cdot -y)-K_{i}(x_0-y)) V_iV_j(\pa_j\rho_2)(y)dy+  \widehat q_{x_0}(t)
\\
&=\cp_{x_0}(x,t) + \td q_{x_0}(t) +  \widehat q_{x_0}(t). 
}
Thus we have $p(x,t) =\cp_{x_0}(x,t) + q_{x_0}(t)$ with
\[
q_{x_0}(t)  = c_{x_0}(t) + \td q_{x_0}(t) +  \widehat q_{x_0}(t). 
\]
Note that using $\norm{G}_{U^{\infty,1}_T}\le \norm{V}_{U^{\infty,2}_T}^2$ and $|x_0-y|>2$ for $y\in \supp(\pa_j\rho_2)$, we have
\EQ{\label{q.est}
\norm{\td q_{x_0}(t)}_{L^{\infty}(0,T)} 
+\norm{\widehat q_{x_0}(t)}_{L^{\infty}(0,T)}
&\lec \norm{\int_{B(x_0,3)\setminus B(x_0,2)}  |G_{ij}|(y) dy}_{L^{\infty}(0,T)}\\
&\lec \norm{G}_{L^{\infty}(0,T;L^1(B(x_0,3)))} \lec \norm{V}_{U^{\infty,2}_T}^2.
}
Since $\td q_{x_0}(t) +  \widehat q_{x_0}(t)$ is in $L^{3/2}(0,T)$, so is $q_{x_0}(t)$.
\end{proof}

Although $\nb V$ has spatial decay, it is not uniform in $t$. Thus, to show the spatial decay of $w$, we will first show \eqref{eq1.5}, i.e., the smallness of $w$ in $L^2_\uloc$ at far distance for a short time in Lemma \ref{th0730b}. For that we need Lemmas \ref{adm}, \ref{th0803} and \ref{th:LEI.w.0}.
\begin{lemma}\label{adm}
For $u_0 \in L^3(\R^3)$, if $\frac 2s + \frac 3q = 1$ and $3\le q< 9$, then
\[
\norm{ e^{t \De}u_0}_{L^s(0,\infty;L^q(\R^3))} \leq C_q \norm{u_0}_{L^3(\R^3)}.
\]
\end{lemma}

This is proved in Giga \cite[196--197]{Gi86}. The case $q=9$ is also true according to 
\cite[Acknowledgment]{Gi86},  but 
there is no detailed proof.

\begin{lemma}\label{th0803}
Suppose $u_0 \in L^2_\uloc$ and $u_0 \in L^3(B(x_0,3))$. Then, $V =e^{t\De}u_0$ satisfies
\EQ{\label{th0803-1}
\norm{ V}_{L^8(0,T; L^4(B(x_0,\frac 32)))}
\lec  \norm{ u_0}_{L^3(B(x_0,3))} + T^{\frac 18} \norm{u_0}_{L^2_\uloc} .
}
\end{lemma}

\begin{proof}
Let $\phi (x)=\Phi(\frac{x-x_0}2)$. Decompose
\[
u_0= u_0\phi + u_0(1-\phi) =: u_1 + u_2. 
\]

By Lemma \ref{adm},
\EQ{\label{est.v1}
\norm{e^{t\De}u_1}_{L^8(0,T; L^4(B(x_0,\frac32)))} 
\leq \norm{ e^{t\De} u_1}_{L^8(0,T;L^4(\R^3))} 
\lec \norm{u_1}_{L^3(\R^3)} 
\le  \norm{u_0}_{L^3(B(x_0,3))}.
}

On the other hand, we have
\EQN{
\norm{e^{t\De}u_2}_{L^8(0,T; L^4(B(x_0,\frac32)))} 
&\lec  \norm{\na e^{t\De}u_2}_{L^8(0,T; L^2(B(x_0,\frac32)))}
+\norm{e^{t\De}u_2}_{L^8(0,T; L^2(B(x_0,\frac32)))}.
}
Obviously, 
\[
\norm{e^{t\De}u_2}_{L^8(0,T; L^2(B(x_0,\frac32)))}
\lec T^\frac 18 \norm{e^{t\De}u_2}_{L^\infty(0,T; L^2_\uloc)}
\lec T^\frac 18 \norm{u_2}_{L^2_\uloc}.
\]
Using $\supp(u_2)\subset B(x_0,2)^c$ and heat kernel estimate, we get
\EQN{
\norm{\na e^{t\De}u_2}_{L^8(0,T;L^2(B(x_0,\frac 32)))}
&\lec T^\frac 18 \norm{\na e^{t\De}u_2}_{L^\infty((0,T)\times B(x_0,\frac32))}
\\
&\lec T^{\frac 18}\int_{B(x_0,2)^c}\frac 1{|x_0-y|^4} |u_0(y)|dy\\
&\lec T^{\frac 18}\sum_{k=1}^\infty \int_{B(x_0,2^{k+1})\setminus B(x_0,2^{k})}\frac 1{{2^{4k}}} |u_0(y)|dy\\
&\lec T^{\frac 18} \norm{u_0}_{L^2_\uloc} .
}
Therefore, we obtain
\[
\norm{e^{t\De}u_2}_{L^8(0,T;L^4(B(x_0,\frac 32)))} \lec T^\frac 18\norm{u_0}_{L^2_\uloc}.
\]
Together with \eqref{est.v1}, we get
\eqref{th0803-1}.
\end{proof}

The perturbation $w=v-V$, $V(t)=e^{t\De}u_0$, satisfies the {\it perturbed Navier-Stokes equations} in the sense of distributions,
\begin{equation}\label{PNS}%
\begin{cases}
\pa_t w -\De w + (V+w)\cdot \na (V+w) + \na p = 0   \\ 
\div w =0 \\
w|_{t=0}=w_0 .
\end{cases}
\end{equation}
It also satisfies the following local energy inequality for test functions supported away from $t=0$.

\begin{lemma}[Local energy inequality for $w$]\label{th:LEI.w.0}
Let $v_0 , u_0 \in L^2_{\uloc,\si}$.
  Let $(v,p)$ be any local energy solution of \eqref{NS} with initial data $v_0$ in $\R^3 \times (0,T)$, $0<T<\infty$.
Then $w(t)=v(t)-e^{t\De}u_0$ satisfies
\EQ{ \label{LEI.w}
\int & |w|^2 \ph(x,t)\, dx
+ 2\int_{0}^t\!\int |\na w|^2 \ph\, dxds \\ 
& \leq %
\int_{0}^t\!\int |w|^2 (\pd_s\ph+\De \ph +v\cdot \na\ph )\,dxds\\
&\quad + \int_{0}^t\!\int 2 p w \cdot \na \ph\, dxds
+ \int_{0}^t\!\int  2V\cdot (v \cdot \na )(w\ph)\, dxds,
}
for any non-negative $\ph \in C_c^\infty((0,T)\times \R^3)$
and any $t\in (0,T)$.%
\end{lemma}
Note that $\ph$ vanishes near $t=0$. If $\ph$ does not vanish near $t=0$, the last integral in \eqref{LEI.w} may not be defined.

\begin{proof}
Recall that we have the local energy inequality \eqref{LEI} for $(v,p)$. The equivalent form for $(w,p)$ is
exactly 
\eqref{LEI.w}. 
Indeed, \eqref{LEI} and \eqref{LEI.w} are equivalent because they differ by an equality which is the sum of the weak form of $V$-equation 
with $2v \ph$ as the test function and the weak form of the $w$-equation \eqref{PNS} with $2V \ph$ as the test function, after suitable integration by parts.
This equality can be proved because $V$ and $\nb V$ are in $L^\infty_\loc(0,T; L^\infty(\R^3))$, and $\ph$ has a compact support in space-time.
\end{proof}

For $r>0$, let 
\[
\chi_r(x) = 1 - \Phi\bke{\frac xr},
\]
so that $\chi_r(x)=1$ for $|x|\ge 2r$ and $\chi_r(x)=0$ for $|x|\le r$.

\begin{lemma}\label{th0730b}
Let $v_0 = w_0+u_0$ with $ w_0\in E^2_\si$ and $u_0\in L^3_{\uloc,\si}$.
  Let $(v,p)$ be any local energy solution of \eqref{NS} with initial data $v_0$ in $\R^3 \times (0,T)$, $0<T<\infty$.
Then, there exist $T_0=T_0(\norm{v_0}_{L^2_\uloc}) \in (0,1)$ and 
 $C_0=C_0(\norm{w_0}_{L^2_\uloc}, \norm{u_0}_{L^3_\uloc})>0$ such that
$w(t)=v(t)-e^{t\De} u_0$ satisfies
\EQ{\label{th0730b-2}
 \norm{w(t)\chi_R}_{L^2_\uloc} \leq C_0 (t^\frac 1{20} + \norm{w_0\chi_R}_{L^2_\uloc}),
}
for any $R>0$ and any $t \in (0,T_1)$, $T_1 =\min(T_0,T)$.
\end{lemma}

In this lemma, we do not assume the oscillation decay.

\begin{proof}
By Lemma \ref{lemma24} and similar to \eqref{uni.est.ep.li}, we can find $T_0 = T_0 (\norm{v_0}_{L^2_\uloc}^2)\in (0,1)$ such that, for $T_1=\min(T_0,T)$,
\EQN{
\norm{w}_{\cE_{T_1}} +\norm{V}_{\cE_{T_1}}
\lec \norm{w_0}_{L^2_\uloc}+\norm{u_0}_{L^2_\uloc}. 
}
By interpolation, it follows that for any $2\leq s\leq \infty$, and $2\leq q\leq 6$ satisfying $\frac 2s+\frac 3q = \frac 32$, we have
\EQN{\label{est.103}
\norm{w}_{U^{s,q}_{T_1}} + \norm{V}_{U^{s,q}_{T_1}} 
\lec \norm{w_0}_{L^2_\uloc}+\norm{u_0}_{L^2_\uloc}.
}
On the other hand, by Lemma \ref{th0803}, for any $t\in(0,1)$, %
\EQN{\label{V-strong-est}
\norm{V}_{U^{8,4}_t}\lec \norm{u_0}_{L^3_\uloc}.
}
Let $A = \norm{w_0}_{L^2_\uloc}+\norm{u_0}_{L^3_\uloc}$. Then, both inequalities can be combined for $t\le T_1$ as
\EQ{\label{bdd.A}
\norm{w}_{\cE_{t}} +\norm{V}_{\cE_{t}}
+\norm{w}_{U^{\frac{10}3,\frac{10}3}_{t}} + \norm{V}_{U^{\frac{10}3,\frac{10}3}_{t}}+	\norm{V}_{U^{8,4}_t}\lec A.
}

Fix $x_0\in\R^3$ and $R>0$, and let
\EQ{\label{xi.def}
\phi_{x_0}= \Phi(\cdot - x_0), \quad \xi = \phi_{x_0}^2 \chi_R^2.
}
Fix $\Theta\in C^\infty(\R)$, $\Theta' \ge 0$, $\Theta(t)=1$ for $t>2$, and $\Theta(t)=0$ for $t<1$.
Define  $\th_{\ep}\in C_c^\infty(0,T)$ for sufficiently small $\ep>0$ by
\EQ{\label{th_ep.def}
\th_\ep(s) = \Theta\left(\frac{s}{\ep}\right) - \Theta\left(\frac{s-T+3\ep}{\ep}\right).
}
Thus $\th_{\ep}(s)=1$ in $(2\ep,T-2\ep)$ and $\th_{\ep}(s)=0$ outside of $(\ep,T-\ep)$.
We now consider the local energy inequality \eqref{LEI.w} for $w$ with $\ph(x,s) = \xi(x)\th_\ep(s)$. We may replace $p$ by $\hp_{x_0}$ in \eqref{LEI.w} as supp\,$\xi \subset B(x_0,\frac32)$ and $\iint c_{x_0}(t) w\cdot \nb \xi\, dxdt = 0$.
We now take $\ep \to 0_+$. Since $\norm{v(t)-v_0}_{L^2(B_2(x_0))}\to 0$ and $\norm{V(t)-u_0}_{L^2(B_2(x_0))}\to 0$ as $t\to 0^+$, we get
\EQ{\label{app.w0}
\int_0^{2\ep}\int |w|^2 \xi (\th_\e)' \,dxds \to \int |w_0|^2 \xi dx.
}
The last term in \eqref{LEI.w} converges by Lebesgue dominated convergence theorem using
\EQN{
\int_{0}^t\!\int  |V\cdot (v \cdot \na )(w\xi)|\, dxds 
&\lec 
\norm{V}_{L^8(0,T;L^4(B(x_0,\frac32)))}\norm{v}_{U^{8/3,4}_T}(\norm{\nb w}_{U^{2,2}_T}+\norm{w}_{U^{2,2}_T}),
}
where the right hand side of the inequality is bounded independently of $\ep$.

In the limit $\ep \to 0_+$, for any $t\in (0,T)$, we get
\EQ{ \label{LEI.w.0}
\int & |w|^2(x,t) \xi(x)\, dx
+ 2\int_{0}^t\!\int |\na w|^2 \xi\, dxds \\ 
& \leq \int |w_0|^2 \xi\, dx +\int_{0}^t\!\int |w|^2 (\De \xi +v\cdot \na\xi )\,dxds\\
&\quad + \int_{0}^t\!\int 2 \hp_{x_0} w \cdot \na \xi\, dxds
+ \int_{0}^t\!\int  2V\cdot (v \cdot \na )(w\xi)\, dxds,
}
for $\xi$ given by \eqref{xi.def}. Now, we consider $t\leq T_1$.
Using \eqref{bdd.A}, we have
\[
\int_{0}^t\!\int |w|^2 \De \xi dxds
\lec \norm{w}_{U^{2,2}_t}^2 \lec A^2 t, 
\]
\EQN{
\int_{0}^t\! \int |w|^2(v\cdot \na)\xi dxds
&\lec \norm{v}_{U^{3,3}_t}\norm{w}_{U^{3,3}_t}^2 \lec A^3 t^{\frac 1{10}}.
}
For the convenience, we suppress the indexes $x_0$ and $R$ in $\phi_{x_0}$, $\hp_{x_0}$ and $\chi_R$. 
By additionally using \eqref{est.hp2}, 
\EQN{
\int_{0}^t\! \int \hp w \cdot \na \xi dxds 
&\lec \int_{0}^t\! \int_{B(x_0,\frac 32)} |\hp||w|   dxds
\lec \norm{\hp }_{L^{\frac 32}([0,t]\times B(x_0,\frac 32))}
\norm{w}_{U^{3,3}_t}\\
&\lec \norm{v}_{U^{3,3}_t}^2 \norm{w}_{U^{3,3}_t} \lec A^3 t^{\frac 1{10}}.
}

To estimate the last term in \eqref{LEI.w.0}, %
we decompose it as
\EQN{
&\int_{0}^t\!\int  V\cdot (v \cdot \na )(w\xi)\, dxds
= I_1+I_2+ I_3
\\
&= \int_{0}^t\! \int \xi V\cdot (V \cdot \na )w\, dxds
+\int_{0}^t\! \int \xi V\cdot (w \cdot \na )w\, dxds
+\int_{0}^t\! \int  V\cdot w (v \cdot \na )\xi\,  dxds.
}

We have
\EQN{
|I_1|\lec
\norm{V}_{L^4(0,T;L^4(\supp(\xi)))} ^2 
\norm{\na w}_{U^{2,2}_t}
\lec A^3 t^\frac14 .
}

On the other hand, by Poincar\'{e} inequality, we have
\EQN{
\int_0^t \norm{w\phi\chi}_{L^6}^2 ds
&\lec \int_0^t \norm{\na (w\phi\chi)}_{L^2}^2 ds 
+ \int_0^t \norm{w\phi\chi}_{L^2}^2 ds\\
&\lec \int_0^t \norm{|\na w|\phi\chi}_{L^2}^2 ds
+ \norm{w}_{U^{2,2}_t}^2,
}
which follows that (using Young's inequality $abc \le \e a^2 + \e b^{8/3} + C(\e) c^8$)
\EQN{
|I_2|
\leq& \int_0^t \norm{|\na w|\phi\chi}_{L^2} \norm{w\phi\chi}_{L^4} \norm{V}_{L^4(\supp(\xi))} ds\\
\le& \int_0^t \norm{|\na w|\phi\chi}_{L^2}\norm{w\phi\chi}_{L^6}^\frac34\norm{w\phi\chi}_{L^2}^\frac14  \norm{V}_{L^4(\supp(\xi))} ds\\
\leq& \
\e \int_0^t \bke{\norm{|\na w|\phi\chi}_{L^2}^2 + \norm{w\phi\chi}_{L^6}^2 }ds\\
&+ C(\e) \int_0^t \norm{ V}_{L^4(\supp(\xi))}^8\norm{w\phi\chi}_{L^2}^2ds\\
\le& \ \frac 1{100}\int_0^t \norm{|\na w|\phi\chi}_{L^2}^2 ds + A^2 t 
+ C\int_0^t \norm{V}_{L^4(\supp(\xi))}^8\norm{w\phi\chi}_{L^2}^2ds
}
by choosing suitable $\e$.
It is easy to control $I_3$:
\[
|I_3| \lec t^{\frac 1{10}}\norm{V}_{U^{\frac {10}3,\frac{10}3}_t}\norm{v}_{U^{\frac {10}3,\frac{10}3}_t}\norm{w}_{U^{\frac {10}3,\frac{10}3}_t}\lec{A^3} t^{\frac 1{10}}.
\]

Therefore, we obtain
\EQN{
\abs{\int_{0}^t\!\int  V\cdot (v \cdot \na )(w\xi)\, dxds}
\leq& \norm{|\na w| \phi\chi }_{L^2([0,t]\times \R^3)}^2\\ &+C(1+A^3)\bke{t^\frac 1{10}
+ \int_0^t \norm{V}_{L^4(\supp(\xi))}^8\norm{w\phi\chi}_{L^2}^2ds},
}
for some absolute constant $C$. Finally, we combine all the estimates to get from \eqref{LEI.w.0} that
\EQN{
&\norm{w(t)\phi \chi}_{L^2(\R^3)}^2
+\norm{|\na w| \phi\chi }_{L^2([0,t]\times \R^3)}^2
\\
&\qquad \lec \norm{w_0\chi_R}_{L^2_\uloc}^2+  
(1+A^3)\bke{t^\frac 1{10}
	+ \int_0^t \norm{V}_{L^4(\supp(\xi))}^8\norm{w\phi\chi}_{L^2}^2ds}
}

Note that $\norm{w(t)\phi \chi}_{L^2(\R^3)}^2$ is lower semicontinuous in $t$ as $w\phi$ is weakly $L^2$-continuous in $t$.
By Gr\"{o}nwall's inequality and \eqref{bdd.A}, we have
\EQN{
\norm{w(t)\phi \chi}_{L^2(\R^3)}^2 
\leq C_0^2(\norm{w_0\chi_R}_{L^2_\uloc}^2+ t^\frac1{10}),
}
for some $C_0=C_0(A)>0$. Taking supremum in $x_0$, we get
\[
\norm{w(t)\chi_R}_{L^2_\uloc} \leq C_0 (t^{\frac1{20}} +\norm{w_0\chi_R}_{L^2_\uloc}) .
\]
This finishes the proof of Lemma \ref{th0730b}.
\end{proof}

\begin{lemma}[Strong local energy inequality]\label{th:SLEI}

Let $(v,p)$ be a local energy solution in $\R^3\times (0,T)$ to Navier-Stokes equations \eqref{NS} for the initial data $v_0\in L^2_\uloc$ constructed in Theorem \ref{loc.ex}, as the limit of approximation solutions $(v^{(k)},p^{(k)})$
 of \eqref{reg.NS}. Then there is a subset $\Si \subset (0,T)$ of zero Lebesgue measure such that, for any $t_0 \in (0,T) \setminus \Si$ and any $t \in (t_0,T)$, we have
\EQ{\label{SLEI-v}
\int & |v|^2 \ph(x,t)\, dx
+ 2\int_{t_0}^t\!\int |\na v|^2 \ph\, dxds \\ 
& \leq \int |v|^2 \ph (x,t_0)\,  dx
+\int_{t_0}^t\!\int \bket{ |v|^2 (\pd_s \ph +\De \ph) +(|v|^2+2 p)v\cdot \na\ph } dxds,
}
for any $\ph \in C^\infty_c(\R^3 \times [t_0,T))$. If, furthermore, for some $u_0\in L^2_{\uloc,\si}$, $V(t) = e^{t\De}u_0$ and $w=v-V$, then for any $t_0 \in (0,T) \setminus \Si$ and any $t \in (t_0,T)$, we have
\EQ{\label{pre.decay}
\int & |w|^2 \ph(x,t)\, dx
+ 2\int_{t_0}^t\!\int |\na w|^2 \ph\, dxds \\ 
& \leq \int |w|^2 \ph (x,t_0)\, dx
+\int_{t_0}^t\!\int |w|^2 (\pd_s\ph+\De \ph +v\cdot \na\ph )dxds\\
&\quad + \int_{t_0}^t\!\int 2 p w \cdot \na \ph\, dxds
- \int_{t_0}^t\!\int  2(v \cdot \na )V\cdot w\ph\, dxds,
}
for any $\ph \in C^\infty_c(\R^3 \times [t_0,T))$.
\end{lemma}

This lemma is not for general local energy solutions, but only for those constructed by the approximation \eqref{reg.NS}. 
Also note that \eqref{SLEI-v} is true for $t_0=0$ since it becomes \eqref{LEI}, but 
\eqref{pre.decay} is unclear for $t_0=0$ since the last integral in \eqref{pre.decay} may not be defined without further assumptions; Compare it with \eqref{LEI.w.0}.

\begin{proof} For any $n \in \NN$,
the approximation $v^{(k)}$ satisfy
\[
\lim_{k \to \infty} \norm{v^{(k)} -v}_{L^2(0,T; L^2(B_n))} =0.
\]
Thus there is a set $\Si_n \subset (0,T)$ of zero Lebesgue measure such that 
\[
\lim_{k \to \infty} \norm{v^{(k)}(t) -v(t)}_{L^2(B_n)} =0, \quad \forall t \in [0,T)\setminus \Si_n.
\]

Let
\[
\Si = \cup _{n=1}^\infty \Si_n, \quad |\Si|=0.
\]
We get
\EQ{\label{eq4.14}
\lim_{k \to \infty} \norm{v^{(k)}(t) -v(t)}_{L^2(B_n)} = 0, \quad \forall t \in [0,T)\setminus \Si, \ \forall n \in \NN.
}
The local energy equality of $(v^{(k)},p^{(k)})$ in $[t_0,T]$ is derived similarly to \eqref{LEI.vep00}
\EQ{\label{SLEI-vk}
2\int_{t_0}^T\! \int |\na v^{(k)}|^2\psi dxds
=& \int  |v^{(k)}|^2 \psi (x,t_0)\,  dx +\int_{t_0}^T\!\int |v^{(k)}|^2(\pa_s \psi +\De \psi) dxds \\
&+\int_{t_0}^T\!\int |v^{(k)}|^2\Phi_{(k)} ({\cal J}_{(k)}(v^{(k)}) \cdot \na)\psi dxds\\
&+ \int_{t_0}^T\! \int 2 p^{(k)} v^{(k)} \cdot \na \psi \, dxds\\
&-\int_{t_0}^T\!\int |v^{(k)}|^2\psi ({\cal J}_{(k)}(v^{(k)})\cdot \na)\Phi_{(k)}dxds,
}
for any $\psi \in C^\infty_c(\R^3 \times [0,T))$. By \eqref{eq4.14}, 
we have
\[
\lim_{k \to \infty} \int |v^{(k)}|^2 \psi(x,t_0)\, dx = \int |v|^2 \psi (x,t_0)\, dx 
\]
for $t_0 \in [0,T)\setminus \Si$. Taking limit infimum $k \to \infty$ in \eqref{SLEI-vk}, we get 
\EQN{%
&2\int_{t_0}^T\!\int |\na v|^2 \psi\, dxds  
\\
& \leq \int |v|^2 \psi (x,t_0)\,  dx
+\int_{t_0}^T\!\int \bket{ |v|^2 (\pd_s \psi +\De \psi) +(|v|^2+2 p)v\cdot \na\psi } dxds.
}
By the same argument for \eqref{pre.lei}, we get \eqref{SLEI-v} from the above. 

Finally, inequality \eqref{pre.decay}  for $t_0>0$ is equivalent to \eqref{SLEI-v} by the same argument of Lemma \ref{th:LEI.w.0}. We have integrated by parts the last term in \eqref{pre.decay}, which is valid since $\nb V \in L^\infty(\R^3 \times (t_0,t))$.
\end{proof}

\medskip

We now prove the main result of this section.

\begin{proposition}[Decay of $w$ and $\cp$]\label{dot.E.3} 
Let $v_0 = w_0+u_0$ with $ w_0\in E^2_\si$ and $u_0\in L^3_{\uloc,\si}$,
and
\EQN{
\lim_{|x_0|\to \infty}\int_{B(x_0,1)}| v_0- (v_0)_{B(x_0,1)}| dx =0.
}
Let $(v,p)$ be a local energy solution in $\R^3\times (0,T)$ to Navier-Stokes equations \eqref{NS} for the initial data $v_0\in L^2_\uloc$ constructed in Theorem \ref{loc.ex}, as the limit of approximation solutions $(v^{(k)},p^{(k)})$ of \eqref{reg.NS}. 
Let $w= v-V$ for $V(t) =e^{t\De}u_0$.
Then, $w$ and $\cp_{x_0}$, defined in Lemma \ref{new.decomp}, decay at spatial infinity: For any $t_1\in (0,T)$,
\EQ{\label{dot.E.3-decay} 
\lim_{|x_0|\to \infty}\bke{
\norm{w}_{L^\infty_t L^2_x \cap L^3(Q_{x_0})} + \norm{\nb w}_{L^2(Q_{x_0})}
+ \norm{\cp_{x_0}}_{L^{\frac 32}(Q_{x_0})} 
}=0,
}
where $Q_{x_0} = B(x_0,\frac 32)\times(t_1,T) $.

\end{proposition}

Note that we do not assert uniform decay up to $t_1=0$. We assume the approximation \eqref{reg.NS} only to ensure the conclusion of Lemma \ref{th:SLEI}, the strong local energy inequality.

\begin{proof}
Choose $A=A(\norm{w_0}_{L^2_\uloc},\norm{u_0}_{L^2_\uloc},T)$ such that
\EQN{
\norm{w}_{\cE_T} +\norm{V}_{\cE_T}
+
\norm{w}_{U^{s,q}_T} + \norm{V}_{U^{s,q}_T} 
&\lec A,
}
for any $2\leq s\leq \infty$, and $2\leq q\leq 6$ satisfying $\frac 2s+\frac 3q = \frac 32$.

Fix $x_0\in \ZZ^3$ and $R \in \NN$. Let $\phi_{x_0}=\Phi(\cdot-x_0)$, $\chi_R(x)=1-\Phi\left(\frac {x}R\right)$, and 
\EQ{\label{xi.def}
\xi = \phi_{x_0}^2 \chi_R^2.
}
For the convenience, we suppress the indexes $x_0$ and $R$ in $\phi_{x_0}$, $\cp_{x_0}$ and $\chi_R$. 

Let $\Si$ be the subset of $(0,T)$ defined in Lemma \ref{th:SLEI}. For any $ t_0\in (0,t_1)\setminus \Si$ and $t\in (t_0,T)$, choose $\th(t) \in C^\infty_c(0,T)$ with $\th=1$ on $[t_0,t]$. Let $\ph(x,t) = \xi(x) \th(t)$. By \eqref{pre.decay} of  Lemma \ref{th:SLEI}, using $t_0\not \in \Si$, we have
\EQ{\label{eq4.19}
\int |w(x,t)|^2 \xi (x)\, dx
+&2\int_{t_0}^t\!\int |\na w|^2 \xi\, dxds \\ 
\leq &\int |w(x,t_0)|^2 \xi (x) \,dx
+\int_{t_0}^t\!\int |w|^2 (\De \xi +(v\cdot \na)\xi )\,dxds\\
&+ 2\int_{t_0}^t\!\int \cp_{x_0} w \cdot \na \xi \,dxds
-2\int_{t_0}^t\!\int  (v \cdot \na )V\cdot w\xi \,dxds.
}
Above we have replaced $p$ by $\cp_{x_0}$ using $\iint q_{x_0}(t) w \cdot \na \xi \,dxds=0$.

By the choice of $\xi$, we can easily see that
\[
\int |w(\cdot,t)|^2 \xi  dx
+2\int_{t_0}^t\!\int |\na w|^2 \xi dxds
\geq
 \norm{w(\cdot,t)\chi}_{L^2(B(x_0,1))}^2
+2\norm{|\na w| \chi }_{L^2([t_0,t]\times B(x_0,1) )}^2,
\]
\[
\int |w(\cdot,t_0)|^2 \xi dx\lec \norm{w(\cdot,t_0)\chi}_{L^2_\uloc}^2,
\]
\[
\int_{t_0}^t\!\int |w|^2 \De \xi dxds
\lec \norm{w\chi}_{U^{2,2}(t_0,t)}^2 
+ \frac 1R \norm{w}_{U^{2,2}_T}^2,
\]
and
\EQN{
 \int_{t_0}^t\!\int |w|^2(v\cdot \na)\xi dxds
&\lec \norm{v}_{U^{3,3}_T}\norm{w\chi}_{U^{3,3}(t_0,t)}^2 + \frac 1R\norm{v}_{U^{3,3}_T}\norm{w }_{U^{3,3}_T}^2\\
&\lec_A \norm{w\chi}_{U^{3,3}(t_0,t)}^2 + \frac 1R.
}

The last term can be also estimated by
\EQN{
\left|\int_{t_0}^t\!\int (v \cdot \na )V\cdot w\xi dxds\right|
&\lec \norm{|\na V|\chi}_{U^{\infty,3}(t_0,T)}
\norm{v}_{U^{2,6}_T}
\norm{w\chi}_{U^{2,2}(t_0,t)}\\
&\lec_A \norm{w\chi}_{U^{2,2}(t_0,t)}^2  
+ \norm{|\na V|\chi}_{U^{\infty,3}(t_0,T)}^2.
}

The only remaining term is the one with pressure. Note
\EQN{
\int_{t_0}^t\!\int &\cp w \cdot \na \xi dxds 
\lec \int_{t_0}^t\!\int_{B(x_0,\frac 32)} |\cp||w|  \chi^2 dxds
+\frac 1R \int_{t_0}^t\!\int_{B(x_0,\frac 32)} |\cp||w|\chi  dxds   \\
&\lec \norm{\cp \chi}_{L^{\frac 32}([t_0,t]\times B(x_0,\frac 32))}
\norm{w\chi}_{U^{3,3}(t_0,t)}
+ \frac 1R \norm{\cp}_{L^{\frac 32}([0,T]\times B(x_0,\frac 32))}
\norm{w\chi}_{U^{3,3}_T}.
}
For the second term, we can use a bound uniform in $x_0$ 
\EQN{
\norm{\cp_{x_0}}_{L^\frac32([0,t]\times B(x_0,\frac 32))} \le C \norm{v}_{U^{3,3}_t}^2 + C(T) \norm{V}_{U^{\infty,2}_T}^2, %
}
which follows from \eqref{est.hp2}, \eqref{hp.cp} and \eqref{q.est}. For the first term, although the other factor $\norm{w\chi}_{U^{3,3}(t_0,t)}$ also has decay, it is larger than the left side of \eqref{eq4.19} by itself.
Hence we need to estimate $\norm{\cp\chi}_{L^{\frac 32}([t_0,t]\times B(x_0,\frac 32))}$ and show its decay.

Let $F_{ij}=w_iw_j + w_i V_j + w_jV_i$ and $G_{ij}=V_iV_j$. The local pressure $\cp$ defined in Lemma \ref{new.decomp} can be further decomposed as 
\[
\cp(x,t)
=p^F + p^{G,1}+ p^{G,2} + p^{G,3}
\]
where $p^F=p^F_{x_0}$ is defined as in \eqref{pFx0},
\[
p^{G,1}
=\int K_i(\cdot-y) [\pa_jG_{ij}] \rho_2 (y) dy, 
\]
\EQN{
p^{G,2}
=& \int (K_{ij}(\cdot-y)-K_{ij}(x_0-y)) G_{ij}(\rho_\tau-\rho_2)(y)dy\\
&- \int (K_{i}(\cdot-y)-K_{i}(x_0-y)) G_{ij}\pa_j(\rho_\tau- \rho_2)(y)dy,
}
for $\rho_\tau = \Phi\left(\frac{\cdot -x_0}{\tau}\right)$, $\tau>4$, and 
\EQN{
p^{G,3} =& \int (K_{ij}(\cdot-y)-K_{ij}(x_0-y))G_{ij}(1-\rho_\tau)(y)dy \\
&+ \int (K_{i}(\cdot-y)-K_{i}(x_0-y)) G_{ij}\pa_j\rho_\tau(y)dy.
}

Recall $p^F=p^F_{x_0}$ 
\begin{align*}
p^F &= -\frac 13 \tr F 
+ \pv \int_{B(x_0,2)}  K_{ij}(\cdot-y) F_{ij}(y) dy
\\
&\quad + \int_{B(x_0,2)^c}  (K_{ij}(\cdot-y)-K_{ij}(x_0-y)) F_{ij}(y) dy\\
&= p^{F,1} +p^{F,2} + p^{F,3}.
\end{align*}
We estimate $p^{F,i}\chi$, $i=1,2,3$. Obviously, we have
\EQN{
\norm{p^{F,1}\chi}_{L^{\frac 32}([t_0,t]\times B(x_0,\frac 32))}
\lec \norm{F\chi}_{U^{\frac 32,\frac 32}(t_0,t)}.
}
Using $L^p$-norm preservation of Riesz transfroms and $\norm{\na \chi}_\infty\lec \frac 1R$,
\EQN{
\norm{p^{F,2}\chi}_{L^{\frac 32}([t_0,t]\times B(x_0,\frac 32))}
\le& \norm{\pv \int_{B(x_0,2)}  K_{ij}(\cdot-y) F_{ij}(y)\chi(y) dy }_{L^{\frac 32}([t_0,t]\times B(x_0,\frac 32))}\\
&+ \norm{\pv \int_{B(x_0,2)}  K_{ij}(\cdot-y) F_{ij}(y)(\chi(\cdot)-\chi(y)) dy}_{L^{\frac 32}([t_0,t]\times B(x_0,\frac 32))}\\
\lec& \norm{F\chi}_{U^{\frac 32,\frac 32}(t_0,t)} 
+ \frac 1R \norm{ \int_{B(x_0,2)}  \frac 1{|\cdot-y|^2} |F_{ij}(y)| dy}_{L^{\frac 32}(t_0,t;L^3 (\R^3)))}\\
\lec& \norm{F\chi}_{U^{\frac 32,\frac 32}(t_0,t)} 
+ \frac 1R \norm{F}_{U^{\frac 32,\frac 32}(t_0,t)}.
}
The last inequality follows from the Riesz potential estimate. Since
\[
|\chi(x)-\chi(y)|\leq \norm{\na\chi}_\infty |x-y| \lec \frac 1{\sqrt{R}}
\]
for $x\in B(x_0,\frac 32)$ and $y\in B(x_0,\sqrt R)$, and
\[
|x-y| \geq |x_0-y| - |x-x_0| \geq \frac 14 |x_0-y|
\]
for $x\in B(x_0, \frac 32)$ and $y\in B(x_0,2)^c$, we get
\EQN{
\norm{p^{F,3}\chi}_{L^{\frac 32}([t_0,t]\times B(x_0,\frac 32))}
\le & \ \norm{\int_{B(x_0,\sqrt{R})\setminus B(x_0,2)}  \frac1{|\cdot-y|^4} F_{ij}\chi(y) dy}_{L^{\frac 32}([t_0,t]\times B(x_0,\frac 32))}\\
& + \norm{\int_{B(x_0,\sqrt{R})\setminus B(x_0,2)}  \frac1{|\cdot-y|^4} F_{ij}(y) (\chi(\cdot)-\chi(y))dy}_{L^{\frac 32}([t_0,t]\times B(x_0,\frac 32))}\\
& + \norm{\int_{B(x_0,\sqrt{R})^c} \frac1{|\cdot-y|^4} F_{ij}(y) dy\chi}_{L^{\frac 32}([t_0,t]\times B(x_0,\frac 32))}.
}
Thus
\EQN{
\norm{p^{F,3}\chi}_{L^{\frac 32}([t_0,t]\times B(x_0,\frac 32))}
\lec & \ \sum_{k=1}^\infty \norm{\int_{B(x_0,2^{k+1})\setminus B(x_0,2^k)}  \frac1{|x_0-y|^4} |F_{ij}\chi(y)| dy}_{L^{\frac 32}(t_0,t;L^\infty( B(x_0,\frac 32)))}\\
& + \frac 1{\sqrt R}\sum_{k=1}^{\infty} \norm{\int_{B(x_0,2^{k+1})\setminus B(x_0,2^k)}  \frac1{|x_0-y|^4} |F_{ij}(y)| dy}_{L^{\infty}([t_0,t]\times B(x_0,\frac 32))}\\
& + \sum_{\lfloor \log_2 \sqrt R \rfloor}^\infty\norm{\int_{B(x_0,2^{k+1})\setminus B(x_0,2^k)} \frac1{|x_0-y|^4} |F_{ij}(y)| dy}_{L^{\frac 32}([t_0,t]\times B(x_0,\frac 32))}\\
\lec & \ \norm{ F\chi}_{L^{\frac 32}([t_0,t]\times B(x_0,\frac 32))}
 + \frac 1{\sqrt{R}} \norm{F}_{U^{\infty,1}(t_0,t)}.
}
Combining the estimates for $p^{F,i}\chi$, $i=1,2,3$, we obtain
\EQN{
\norm{p^F\chi}_{L^{\frac 32}([t_0,t]\times B(x_0,\frac 32))}
&\lec_T \norm{F\chi}_{U^{\frac 32,\frac 32}(t_0,t)} + \frac 1R\norm{F}_{U^{\frac 32,\frac 32}(t_0,t)} + \frac 1{\sqrt{R}}\norm{F}_{U^{\infty,1}(t_0,t)}\\
&\lec_{A,T}\norm{w\chi}_{U^{3,3}(t_0,t)} + \frac 1{\sqrt{R}}.
}

Now, we consider $p^{G,i}$'s. Since for $x\in B(x_0,\frac 32)$, $p^{G,1}$ satisfies
\EQN{
|p^{G,1}\chi(x,t)|
\leq&
\int_{|x_0-y|\leq 3} |(\na K)(x-y)||V||\na V|(y,t)(|\chi(y)|+|\chi(x)-\chi(y)|)dy\\
\lec& 
\int_{B_3(x_0)} \frac 1{|x-y|^2}||V||\na V(y,t)|\chi(y)dy+\frac 1R\int_{B_3(x_0)} \frac 1{|x-y|}|V||\na V(y,t)|dy
}
using $|\chi(x)-\chi(y)| \lec \norm{\nb \chi}_\infty |x-y|$,
the estimate for $p^{G,1}\chi$ can be obtained from Young's convolution inequality;
\EQN{
\norm{p^{G,1}\chi}_{L^{\frac 32}([t_0,t]\times B(x_0,\frac 32))}
\lec&_T 
\norm{\int_{|x_0-y|\leq 3} \frac 1{|\cdot-y|^2}||V||\na V|(y,t)\chi(y)dy}_{L^{2}([t_0,t]\times \R^3)}\\
&+\frac 1R\norm{\int_{|x_0-y|\leq 3} \frac 1{|\cdot-y|}|V||\na V|(y,t)dy}_{L^{\frac {20}{13}}(t_0,t;L^{\frac{30}7}(\R^3))}\\
\lec& \norm{\frac 1{|\cdot|^2}}_{\frac 32, \infty}\norm{|\na V|\chi}_{L^\infty_t(t_0,T;L^\frac 32(B(x_0,3)))}\norm{V}_{L^2(0,T; L^6(B(x_0,3)))}\\
&+ \frac 1R \norm{\frac 1{|\cdot|}}_{3,\infty} \norm{V}_{L^{\frac {20}3}(0,T; L^\frac52 (B(x_0,3)) )}\norm{\na V}_{U^{2,2}_T}\\
\lec&_{A,T} \norm{|\na V|\chi}_{U^{\infty,\frac 32}(t_0,T)}
+\frac 1R. 
}

By integration by parts, for $x \in B(x_0, \frac 32)$, $p^{G,2}$ can be rewritten as
\[
p^{G,2} =\int (K_{i}(\cdot-y)-K_{i}(x_0-y)) V_i\pa_jV_j(y,t)(\rho_\tau-\rho_2)(y)dy
\]
and then it satisfies
\EQN{
|p^{G,2}\chi(x,t)|
&\lec \int_{2<|x_0-y|\leq 2\tau} \frac 1{|x_0-y|^3}|V||\na V|(y,t)(|\chi(y)|+|\chi(x)-\chi(y)|) dy\\
&\lec \sum_{i=1}^{m_\tau} \int_{B_{i+1}\setminus B_i} \frac 1{|x_0-y|^3}|V||\na V|(y,t)\left(|\chi(y)|+\frac {\tau}R\right) dy,
} 
where $m_\tau = \lceil \ln (2\tau)/\ln 2 \rceil$ and $B_i = B(x_0,2^i)$. 
Taking $L^2(t_0,t)$ on it, we have
\EQN{
\norm{p^{G,2}\chi}_{L^2(t_0,t;L^\infty(B(x_0,\frac 32)))}
&\lec\norm{
\sum_{i=1}^{m_\tau} \int_{B_{i+1}\setminus B_i} \frac 1{|x_0-y|^3}|V||\na V|(y,t)\left(|\chi(y)|+\frac {\tau}R\right) dy}_{L^2(t_0,t)}\\
&\lec \sum_{i=1}^{m_\tau} \frac 1 {2^{3i}} \left(\norm{V|\na V|\chi}_{L^2(t_0,t;L^1(B_{i+1}))}
+ \frac {\tau}R\norm{V|\na V|}_{L^2(t_0,t;L^1(B_{i+1}))}\right)\\
&\lec \sum_{i=1}^{m_\tau}\bke{ (\norm{|V||\na V|\chi}_{U^{2,1}(t_0,T)}
+ \frac {\tau}R\norm{|V||\na V|}_{U^{2,1}_T} }\\
&\lec_T \ln \tau \norm{V}_{U^{\infty,2}_T}\norm{|\na V|\chi}_{U^{\infty,2}(t_0,T)}+ \frac {\tau\ln \tau}{ R}\norm{V}_{U^{\infty,2}_T}\norm{\na V}_{U^{2,2}_T}.
}

Lastly,
\[
|p^{G,3}(x,t)| \le \int_{|x_0-y|\geq \tau } \frac {|V(y,t)|^2}{|x_0-y|^4}  dy + \frac 1{\tau}\int_{\tau \leq |x_0-y| \leq 2\tau} \frac {|V(y,t)|^2}{|x_0-y|^3} dy
\le  \frac 1{\tau}\norm{V}_{U^{\infty,2}_T}^2 .
\]
Hence
\[
\norm{p^{G,3}\chi}_{L^\frac 32([t_0,t]\times B(x_0, \frac 32))}
\leq \norm{p^{G,3}}_{L^\frac 32([t_0,t]\times B(x_0, \frac 32))}
\lec_{A,T} \frac 1{\tau}
\]

To summarize, we have shown
\[
\sum_{i=1}^3\norm{p^{G,i}\chi}_{L^{\frac 32}([t_0,t]\times B(x_0,\frac 32))}
\lec_{A,T}\ln \tau \norm{|\na V|\chi}_{U^{\infty,2}(t_0,T)} + \frac {\tau \ln \tau}R + \frac 1 \tau,
\]
and therefore
\EQ{\label{est.decay.p}
\norm{\cp\chi}_{L^{\frac 32}([t_0,t]\times B(x_0,\frac 32))}
\lec_{A,T}\norm{w\chi}_{U^{3,3}(t_0,t)} + \frac 1{\sqrt R} +
\ln \tau \norm{|\na V|\chi}_{U^{\infty,2}(t_0,T)} + \frac {\tau \ln \tau}R + \frac 1 \tau.
}

Finally, combining all estimates and then taking supremum on \eqref{eq4.19} over $x_0\in \R^3$, we obtain
\EQ{\label{pre.decay.est00}
 \norm{w(\cdot,t)\chi}_{L^2_\uloc}^2
+&2\norm{|\na w| \chi }_{U^{2,2}(t_0,t)}^2\\
\lec_{A,T}& \norm{w(\cdot,t_0)\chi}_{L^2_\uloc}^2 
+\norm{w\chi}_{U^{2,2}(t_0,t)}^2 
+\norm{w\chi}_{U^{3,3}(t_0,t)}^2\\ 
&+
(\ln \tau)^2 \norm{|\na V|\chi}_{U^{\infty,3}(t_0,T)}^2 + \frac {(\tau \ln \tau)^2}{R^2} + \frac 1 {\tau^2}+ \frac 1{R}.
}
Using the estimates
\EQ{\label{U33.est}
\norm{w \chi}_{U^{3,3}(t_0,t)}^2
\lec \norm{w \chi}_{U^{6,2}(t_0,t)}
\bke{
\norm{w \chi}_{U^{2,2}(t_0,t)} 
+\norm{|\na w| \chi}_{U^{2,2}(t_0,t)}
+\frac 1R\norm{w}_{U^{2,2}_T}
},
}
and Lemma \ref{th0730b},
it becomes
\EQ{\label{pre.decay.est}
\norm{w(\cdot,t)\chi}_{L^2_\uloc}^2
+&\norm{|\na w| \chi }_{U^{2,2}(t_0,t)}^2\\
\lec_{A,T,C_0}&  \
t_0^{\frac 1{10}}+\norm{w_0\chi}_{L^2_\uloc}^2 
+\norm{w\chi}_{L^6(t_0,t;L^2_\uloc)}^2\\ 
&+
(\ln \tau)^2 \norm{|\na V|\chi}_{U^{\infty,3}(t_0,T)}^2 + \frac {(\tau \ln \tau)^2}{R^2} + \frac 1 {\tau^2}+ \frac 1{R},
}
where $C_0$ is defined as in Lemma \ref{th0730b}.

Note that $\norm{w(\cdot,t)\chi}_{L^2_\uloc}^2$ is lower semicontinuous in $t$ as $w$ is weakly $L^2(B_n)$-continuous in $t$ for any $n$.
By Gr\"{o}nwall inequality,
we have 
\EQ{\label{decay.est}
\norm{w\chi}_{L^6(t_0,T;L^2_\uloc)}^2
\lec_{A,T,C_0}&  \ t_0^{\frac 1{10}}+\norm{w_0\chi}_{L^2_\uloc}^2 
\\ 
&+(\ln \tau)^2 \norm{|\na V|\chi}_{U^{\infty,3}(t_0,T)}^2+ \frac {(\tau \ln \tau)^2}{R^2} + \frac 1 {\tau^2}+ \frac 1{R}.
}

We now prove \eqref{dot.E.3-decay}. 
Fix $t_1\in (0,T)$. %
For every $n\in \NN$ we can choose $t_0=t_0(n)\in (0,t_1)\setminus \Si$ satisfying
\[
t_0^{\frac 1{10}}<\tfrac 1n.
\] 
At the same time, we pick $\tau=\tau(n)>4$ satisfying $\tau^{-2} \leq 1/n$. 
After $t_0$ and $\tau$ are fixed, we can make all the remaining terms small by choosing $R=R(n,\norm{v_0}_{L^2_\uloc}, t_0,\tau)$ sufficiently large:
\[
\norm{w_0\chi_R}_{L^2_\uloc}^2+(\ln \tau)^2 \norm{|\na V|\chi_R}_{U^{\infty,3}(t_0,T)}^2+ \frac {(\tau \ln \tau)^2}{R^2} + \frac 1{R}\leq \frac 1n.
\]
Here, the smallness of the second term follows from $\nb V$ decay (Lemma \ref{decay.na.U}), using the oscillation decay of $v_0$. 
In conclusion, by \eqref{decay.est}, for each $n\in \NN$, we can find $t_0$, $\tau$ and $R\gg 1$ so that
\[
\norm{w\chi_R}_{L^6(t_0,T;L^2_\uloc)}^2
\lec_{A,T,C_0}\frac 1n. 
\]
By \eqref{pre.decay.est},
\[
\norm{w\chi_R}_{L^\infty(t_0,T;L^2_\uloc)}^2
+\norm{|\na w|\chi_R}_{U^{2,2}(t_0,T)}\lec_{A,T,C_0} \frac 1n.
\]
By \eqref{U33.est},
\[
\norm{w\chi_R}_{U^{3,3}(t_0,T)}^2 
\lec_{A,T,C_0} \frac 1n.
\]

Restricted to the original time interval $(t_1,T)$, the perturbation $w$ satisfies
\[
\lim_{R\to \infty}\norm{w\chi_R}_{U^{3,3}(t_1,T)} = 0,
\]
\[
\lim_{R\to \infty}\norm{w\chi_R}_{L^\infty(t_1,T;L^2_\uloc)}^2
+\norm{|\na w|\chi_R}_{U^{2,2}(t_1,T)} =0.
\]
Using \eqref{est.decay.p}, we also have
\[
\lim_{R\to \infty}\sup_{x_0\in \R^3} \norm{\cp_{x_0}\chi_R}_{L^{\frac32}(B(x_0,\frac32)\times(t_1,T))} = 0.
\]
This completes the proof of Proposition \ref{dot.E.3}.
\end{proof}

\begin{corollary}\label{w.E4}
Under the same assumptions of Proposition \ref{dot.E.3}, the perturbed Navier-Stokes flow $w= v-e^{t\De}u_0$ satisfies $w(t)\in E^p(\R^3)$ for almost all $t\in (0,T]$ for any $3\leq p\leq  6.$
\end{corollary}
\begin{proof}
By Proposition \ref{dot.E.3}, for any fixed $x_0\in \R^3$ and $t_1\in (0,T)$, 
the perturbed local energy solution $w$ to Navier-Stokes equations satisfies 
\[
\norm{w}_{L^3(B_{3/2}(x_0) \times (t_1,T))}
+\norm{\cp_{x_0}}_{L^{3/2}(B_{3/2}(x_0) \times (t_1,T))} \to 0 \quad\text{as }|x_0|\to \infty.
\]
Recall that $V\in C^1([\de,\infty)\times \R^3)$ for any $\de>0$. 
Then, by Caffarelli-Kohn-Nirenberg criteria \cite{CKN},
for any $t_2\in (t_1,T]$, we can find $R_0>0$ such that if $|x_0|\geq R_0$, 
 \[
 \norm{w}_{L^\infty([t_2,T] \times B_1(x_0))}
 \lec  \norm{w}_{L^3(B_{3/2}(x_0) \times (t_1,T))}
 +\norm{\cp_{x_0}}_{L^{3/2}(B_{3/2}(x_0) \times (t_1,T))}^{1/2},
 \]
and the constant in the inequality is independent of $x_0$. Moreover, $ \norm{w}_{L^\infty([t_2,T] \times B_1(x_0))} \to 0$ as $|x_0|\to \infty$.  Although the system \eqref{PNS} satisfied by $w$ is not the original \eqref{NS}, similar proof works since $V \in C^1$. See \cite[Theorem 2.1]{JiaSverak} for more singular $V\in L^m$, $m>1$, but without the source term $V \cdot \nb V$.

 On the other hand, $w\in \cE_T$ implies that \[
 w\in L^ s(0,T;L^p(B_{R_0}))
 \]
 for any $s\in [2,\infty]$ and $p\in [2,6]$ with $\frac 2s + \frac 3p = \frac 32$, 
 and therefore $w(t)\in E^p$ for a.e.~$t\in (0,T]$.
\end{proof}

\section{Global existence}\label{global.sec}

In this section, we prove Theorem \ref{global.ex}. We first give the following decay estimates.  

\begin{lemma}\label{dot.E.3.2} 
Let $(v,p)$ be a local energy solution in $\R^3\times [t_0,T]$, $0<t_0<T<\infty$, to the Naiver-Stokes equations \eqref{NS} for the initial data
\[
v|_{t=t_0}=w_* + e^{t_0\De}u_0
\] 
where $w_*\in E^2_\si$ and $u_0 \in L^3_{\uloc,\si}$ satisfies the oscillation decay \eqref{ini.decay}.
Let $V(t)=e^{t\De}u_0$. Then, the perturbation $w= v-V $ also decays at infinity:
\[
\norm{w}_{L^3([t_0,T]\times B(x_0,1) )}+\norm{\cp_{x_0}}_{L^{\frac 32}([t_0,T]\times B(x_0,1) )} \to 0, %
\]
and
\[
\norm{w}_{L^\infty(t_0,T;L^2(B(x_0,1)))}
+\norm{\na w}_{L^2(t_0,T;L^2(B(x_0,1)))} \to 0 ,\quad\text{as }|x_0|\to \infty.
\]
\end{lemma}

{\it Remark.} This $T$ is arbitrarily large, unlike the existence time given in the local existence theorem, Theorem \ref{loc.ex}. We assume $w_*\in E^2$, and we have $V \in C^1(\R^3 \times [t_0,T])$. We no longer need Lemma \ref{th0730b} nor the strong local energy inequality.
\begin{proof}
The proof is almost the same as that of Proposition \ref{dot.E.3} except for the way to estimate $\norm{w(\cdot, t_0)\chi_R}_{L^2_\uloc}$ in \eqref{pre.decay.est00}. Indeed, $\lim_{R \to \infty}\norm{w(\cdot, t_0)\chi_R}_{L^2_\uloc} = 0$ by the assumption $w(\cdot, t_0)=w_*\in E^2$.
\end{proof}

\medskip

Now, we prove the main theorem. 
\begin{proof}[Proof of Theorem \ref{global.ex}.]
Let $(v,p)$ be a local energy solution to the Naiver-Stokes equations in $\R^3\times[0,T_0]$, $0<T_0<\infty$, for the initial data $v|_{t=0}=v_0$, constructed in Theorem \ref{loc.ex}. By Corollary \ref{w.E4}, there exists $t_0\in (0,T_0)$, arbitrarily close to $T_0$, with  $w(t_0)=v(t_0)-e^{t_0\De}u_0\in E^4$. Then, by Lemma \ref{decomp.Ep}, for any small $\de>0$, we can decompose
\[
w(t_0) = W_0 + h_0, 
\]
where $W_0\in C_{c,\si}^\infty(\R^3)$ and $h_0\in E^4(\R^3)$ with $\norm{h_0}_{L^4_\uloc}<\de$.

To construct a local energy solution $(\td v, \td p)$ to \eqref{NS} for $t \ge t_0$ with initial data $\td v|_{t=t_0} = v(t_0)$, we decompose $(\td v, \td p)$ as 
\[
\td v = V + h+ W, \quad \td p = p_h + p_W
\]
where $V(t)=e^{t\De} v_0$, $(h,p_h)$ satisfies 
\EQ{\label{eqn.h}
\begin{cases}
\pa_t h -\De h + \na p_h = -H \cdot \nb H, \quad H = V+h,\\
\div h =0,\quad
h|_{t=t_0} = h_0,
\end{cases}
}
so that $H$ solves \eqref{NS} with $H(t_0)=e^{t_0\De}u_0+h_0$, 
and $(W,p_W)$ satisfies 
\EQ{\label{eqn.td.w}
\begin{cases}
\pa_t W -\De W  + \na p_W 
= -[(H+W)\cdot \na]W -(W \cdot \na)H,  \\
\div W = 0,\quad
W|_{t=t_0} = W_0.  
\end{cases}
}

Our strategy is 
to first find, for each $\ep>0$,
 a distributional solution $(h^\ep,p_h^\ep)$ and a Leray-Hopf weak solution $(W^\ep,p_W^\ep)$ to
$\ep$-approximations of \eqref{eqn.h} and \eqref{eqn.td.w} for $t\in I$ for some $S=S(\de, V)>0$ uniform in $\ep$. Then, we prove that they have a limit $(\td v, \td p)$ which is a desired local energy solution to \eqref{NS} on $I$. By gluing two solutions $v$ and $\td v$ at $t=t_0$, we can get the extended local energy solution on the time interval $[0, t_0+S]$. Repeating this process, we get a time-global local energy solution. The detailed proof is given below.

\bigskip

\noindent\texttt{Step 1.} Construction of approximation solutions

\medskip

Let $I=(t_0,t_0+S)$ for some small $S\in (0,1)$ to be decided.
For $0 < \ep <1$, we first consider the fixed point problem for
\EQ{\label{op.eqn.h}
 \Psi (h)&=e^{(t-t_0)\De}h_0
-\int_{t_0}^t e^{(t-s)\De}\mathbb{P} \na \cdot(\cJ H \otimes H \Phi_\ep)(s) ds,
\quad H = V + h,
}
where $\cJ(H) = H\ast \eta_\ep$ is mollification of scale $\e$ and $\Phi_\ep(x) = \Phi(\ep x)$ is a localization factor of scale $\ep^{-1}$. We will solve for a fixed point $h=h^\ep$ 
in the Banach space
\[
\cF =\cF_{t_0,S} :=\{h \in U^{\infty,4}(I): (t-t_0)^{\frac 38}h(\cdot,t) \in L^\infty(I\times \R^3) \}
\]
for some small $S >0$
with
\[
\norm{h}_{\cF} := \norm{h}_{U^{\infty,4}(I)} + \norm{(t-t_0)^\frac 38h(t)}_{L^\infty(I\times \R^3)}.
\]
Denote $M= \norm{V}_{L^{\infty}(I\times \R^3)}\lec (1+t_0^{-4/3})\norm{v_0}_{L^2_\uloc}$.
By Lemma \ref{lemma23}, we have
\EQN{
\norm{\Psi h}_{U^{\infty,4}(I)}
&\lec \norm{h_0}_{L^4_\uloc} 
+S^\frac 18\norm{h}_{U^{\infty,4}}^2
+ S^\frac 12M\norm{h}_{U^{\infty,4}}
+ S^\frac 12\norm{V}_{L^\infty(I;L^8_\uloc)}^2 
\\
&\lec \norm{h_0}_{L^4_\uloc} 
+S^\frac 18\norm{h}_{\cF}^2
+ S^\frac 12M\norm{h}_{\cF}
+ S^\frac 12M^2 ,
}
and for $t \in I$,
\EQN{
\norm{\Psi h(t)}_{L^\infty(\R^3)}
&\lec (t-t_0)^{-\frac 38}\norm{h_0}_{L^4_\uloc} 
+\int _{t_0}^t |t-s|^{-1/2} \bke{ \norm{h(s)}_{L^\infty}^2 + M^2} ds
\\
& \lec (t-t_0)^{-\frac 38}\norm{h_0}_{L^4_\uloc}+(t-t_0)^{-1/4} \norm{h}_{\cF}^2 + (t-t_0)^{1/2} M^2.
}

Therefore, we get
\EQN{
\norm{\Phi h}_{\cF}&\lec
\norm{h_0}_{L^4_\uloc} 
+S^\frac 18\norm{h}_{\cF}^2
+ S^\frac 12M\norm{h}_{\cF}
+ S^\frac 12M^2 .
}
Similarly we can show
\EQN{
\norm{\Phi h_1-\Phi h_2}_{\cF}&\lec
\bket{ S^\frac 18(\norm{h_1}_{\cF}+\norm{h_2}_{\cF}) + S^\frac 12M} \norm{h_1-h_2}_{\cF} .
}
By the Picard contraction theorem, we can find $S=S(\de,\norm{V}_{L^\infty(I\times\R^3)} )\in (0,1)$ such that
a unique fixed point (mild solution) $h^\ep$ to \eqref{op.eqn.h} exists in $\cF_{t_0,S}$ with
\EQ{\label{est1.hep}
\norm{h^\ep}_{\cF} \leq C\de, \qquad\forall 0<\ep<1.
}
We also have the uniform bound
\EQ{\label{est2.hep}
\norm{h^\ep}_{\cE(I)} \lec \norm{\cJ H^\ep \otimes H^\ep \Phi_\ep}_{U^{2,2}(I)}\lec \norm{h^\ep}_{\cF}^2 + \norm{V}_{U^{4,4}(I)}^2 \lec \de^2+M^2.
}

Now, we define $H^\ep = V+ h^\ep$ and the pressure $p_h^\ep$ by
\EQN{
p_h^\ep &= -\frac 13 \cJ H^\ep \cdot H^\ep \Phi_\ep
+\pv \int_{B_2} K_{ij}(\cdot-y)(\cJ H^\ep )_i H^\ep _j\Phi_\ep(y,t) dy\\
&+\pv \int_{B_2^c} (K_{ij}(\cdot-y)-K_{ij}(-y))(\cJ H^\ep) _i H^\ep _j\Phi_\ep(y,t) dy.
}
It is well defined thanks to the localization factor $\Phi_\ep$.
For each $R>0$, we have a uniform bound
\EQ{\label{est.phep}
\norm{p_h^\ep}_{L^\frac 32(I\times B_R)}
\le C(R)
}
in a similar way to getting \eqref{uni.bdd.p}. The pair $(h^\ep, p^\ep_h)$ solves, with $H^\ep = V + h^\ep$,
\EQ{\label{eqn.hep}
\begin{cases}
\pa_t h^\ep -\De h^\ep + \na p_h^\ep = 
-(\cJ H^\ep \cdot \na)(H^\ep \Phi_{\ep}) ,\\
\div h^\ep =0,\quad
h^\ep|_{t=t_0} = h_0\in L^4_\uloc
\end{cases}
}
in $\R^3 \times I$ in the distributional sense.

We next consider the equation for $W=W^\ep$,
\EQ{\label{eqn.Wep}
\begin{cases}
\pa_t W -\De W  + \na p_W 
= f_W^\ep
\\
f_W^\ep:=- \cJ (H^\ep+ W)\cdot \nb W - \cJ W \cdot \nb H^\ep,  
\\
\div W = 0,\quad
W|_{t=t_0} = W_0 \in C^\infty_{c,\si}.
\end{cases}
}
Note that \eqref{eqn.Wep} is a mollified and perturbed \eqref{NS}, and has no localization factor $\Phi_\e$.

Using $W^\ep$ itself as a test function, we can get an a priori estimate: for $t\in I$, 
\[
\norm{W(t)}_{L^2(\R^3)}^2 + 2\norm{\na W}_{L^2([t_0,t]\times \R^3)}^2
\le \norm{W_0}_{L^2(\R^3)}^2 + \iint f_W^\ep \cdot W.
\]
Note that $\iint  \cJ (H+ W)\cdot \nb W \cdot W=0$ and $-  \iint  (\cJ W \cdot \nb) h^\ep \cdot W=\iint(\cJ W \cdot \nb) W \cdot h^\ep $. Also recall that
\[
\norm{h^\ep W}_{L^2(Q)}
\lec \norm{h^\ep}_{L^\infty(I;L^3_\uloc)} (\norm{\na W}_{L^2(Q)}
+\norm{ W}_{L^2(Q)})
\]
for $Q=[t_0,t]\times \R^3$. Its proof can be found in \cite[page 162]{KS}. Thus
\EQN{
 \iint f_W^\ep \cdot W 
&= \iint  (\cJ W \cdot \nb) W \cdot h^\ep -  \iint  (\cJ W \cdot \nb) V \cdot W
\\
&\le C\norm{\na W}_{L^2(Q)} \de (\norm{\na W}_{L^2(Q)}
+\norm{ W}_{L^2(Q)})
 + M_1 \norm{W}_{L^2(Q)}^2.
}
where $M_1= \norm{\nb V}_{L^{\infty}(I\times \R^3)}$.
By choosing $\de$ sufficiently small, we conclude
\[
\norm{W(t)}_{L^2(\R^3)}^2 + \norm{\na W}_{L^2([t_0,t]\times \R^3)}^2
\le \norm{W_0}_{L^2(\R^3)}^2 + C(1+M_1)  \norm{W}_{L^2(Q)}^2.
\]
By Gr\"{o}nwall inequality (using that $\norm{W(t)}_{L^2(\R^3)}^2$ is lower semicontinuous), we obtain
\EQ{\label{est.Wep}
\norm{W^\ep}_{L^\infty(I;L^2(\R^3))}^2 &+ \norm{\na W^\ep}_{L^2(I\times \R^3)}^2
\leq C(M_1) \norm{W_0}_{L^2(\R^3)}^2.
}
With this uniform a priori bound, for each $0<\ep<1$, we can use Galerkin method to construct a Leray-Hopf weak solution $W^\ep$ on $I \times \R^3$ to \eqref{eqn.Wep}.

Define $F^\ep_{ij} = \cJ( W^\ep+  H^\ep) _iW^\ep_j + (\cJ W^\ep)_iH^\ep _j $. We have the uniform bound
\[
\norm{F^\ep_{ij}}_{U^{3/2,3/2}(I)} \le C \norm{|V|+|h^\ep|+|W^\ep|}_{U^{3,3}(I)}^2 \le C(M,M_1,\norm{W_0}_{L^2(\R^3)}).
\]
Define
$p^\ep_W (x,t) = \lim_{n \to \infty} p^{\ep,n}_W(x,t) $, and $p^{\ep,n}_W(x,t) $ is defined for $|x|<2^n$ by
\EQN{
p^{\ep,n}_W(x,t) 
=& -\frac 13 \tr F^\ep_{ij}(x,t) +\pv \int_{B_2(0)} K_{ij}(x-y)F^\ep_{ij} (y,t) dy\\
&+\bke{\pv \int_{B_{2^{n+1}}\setminus B_2} + \int_{B_{2^{n+1}}^c}}  (K_{ij}(x-y)-K_{ij}(-y))F^\ep_{ij}(y,t) dy.
}
For each $R>0$, we have a uniform bound
\EQ{\label{est.pWep}
\norm{p_W^\ep}_{L^\frac 32(I\times B_R)}
\le C(R,M,M_1,\norm{W_0}_{L^2(\R^3)}). 
}
By the usual theory for the nonhomogeneous Stokes system in $\R^3$, the pair $(W^\ep,p^\ep_W)$ solves \eqref{eqn.Wep}  in distributional sense.

We now define 
\[
 v^\ep = H^\ep + W^\ep = V + h^\ep + W^\ep,\quad
  p^\ep = p_h^\ep + p_W^\ep. 
\]
Summing \eqref{eqn.hep} and  \eqref{eqn.Wep},
the pair $(v^\ep,p^\ep)$ solves in distributional sense
\EQ{\label{eqn.vep}
\begin{cases}
\pa_t  v^\ep-\De v^\ep  + \na p^\ep
=- \cJ v^\ep \cdot \nb v^\ep  + E^\ep,  
\\
\hspace{26.5mm}
E^\ep = \cJ H^\ep \cdot \nb (H^\ep(1-\Phi_\ep)),
\\
\div v^\ep = 0,\quad
v^\ep|_{t=t_0} = v(t_0).
\end{cases}
}
Thanks to the mollification, $h^\ep$ and $W^\ep$ have higher local integrability by the usual regularity theory. Thus we can test \eqref{eqn.vep} by $2v^\ep\xi$, %
$\xi \in C^\infty_c([t_0,t_0+S) \times \R^3)$, and integrate by parts to get the identity
\EQ{\label{LEI.vep}
& 2\int_{I} \!\int |\na v^\ep |^2\xi \,dxds = \int |v|^2 \xi (x,t_0)\,dx \\
&\quad + \int_{I} \!\int |v^\ep |^2(\pa_s\xi + \De \xi) + (|v^\ep |^2\cJ v^\ep +2p^\ep  v^\ep )\cdot \na \xi +E^\ep \cdot 2v^\ep \xi\,dxds .
}
Note that $v$ in $\int |v|^2 \xi (x,t_0)\,dx$ is the original solution in $[0,T)$.

\bigskip

\noindent\texttt{Step 2.} A local energy solution on $I=(t_0,t_0+S)$

\medskip

We now show that $(v^\ep,p^\ep)$ has a weak limit $(\td v, \td p)$ which is a local energy solution on $I$. 
Recall the uniform bounds \eqref{est1.hep}, \eqref{est2.hep}, \eqref{est.phep}, \eqref{est.Wep}, and \eqref{est.pWep} for $h^\ep,p_h^\ep,W^\ep$ and $p_W^\ep$.
As in the proof of Theorem \ref{loc.ex}, from the uniform estimates and the compactness argument, we can find a subsequence $(v^{(k)}, p^{(k)})$, $k \in \NN$, from $(v^\ep,  p^\ep)$ which converges to some $(\td v, \td p)$ in the following sense: for each $n\in \NN$,
\EQN{
v^{(k)} &\stackrel{\ast}{\rightharpoonup} \td v \qquad\qquad \text{in } L^\infty(I;L^2(B_{2^n})),  \\
v^{(k)} &\rightharpoonup \td v \qquad\qquad\text{in }L^2(I;H^1(B_{2^n})),\\
v^{(k)}, {\cal J}_{(k)}v^{(k)}  &\rightarrow \td v \qquad\qquad\text{in }L^3(I\times B_{2^{n}}), \\ 
p^{(k)} &\to \td{p} \qquad\qquad \text{in }L^{\frac 32}(I\times B_{2^{n}}), 
}
where
$\td p(x,t) = \lim_{n \to \infty}\td p^n(x,t) $, and $\td p^n(x,t) $ is defined for $|x|<2^n$ by
\EQN{
\td p^n(x,t) 
=&-\frac 13 |\td v(x,t)|^2 +\pv \int_{ B_2}  K_{ij}(x-y) \td v_i\td v_j(y,t) \, dy 
\\
&+\bke{\pv \int_{B_{2^{n+1}}\setminus B_2} + \int_{B_{2^{n+1}}^c}}  (K_{ij}(x-y)-K_{ij}(-y)) \td v_i\td v_j(y,t)  \, dy .
}

Taking the limit of the weak form of \eqref{eqn.vep}, we obtain that
 $(\td v, \td p)$ satisfies the weak form of \eqref{NS} for the initial data $\td v|_{t=t_0} = v(t_0)$. Furthermore, %
 the limit of \eqref{LEI.vep} gives us the local energy inequality: For any 
$\xi \in C^\infty_c([t_0,t_0+S) \times \R^3)$, $\xi \ge 0$, we have 
\EQ{\label{LEI.tdv}
& 2\int_{I} \!\int |\na \td v |^2\xi \,dxds \le \int |v|^2 \xi (x,t_0)\,dx \\
&\quad + \int_{I} \!\int |\td v |^2(\pa_s\xi + \De \xi) + (|\td v |^2  +2\td p)  \td v \cdot \na \xi \,dxds .
}
Here we have used that
$
\iint E^{(k)} \cdot v^{(k)} \xi  
= \iint \mathcal{J}_{(k)} H^{(k)} \cdot \nb (H^{(k)}(1-\Phi_{(k)})) \cdot v^{(k)} \xi =0$
for $k$ sufficiently large. In a way similar to the proof of Theorem \ref{loc.ex}, we get the local pressure decomposition for $\td p$, weak local $L^2$-continuity of $\td v(t)$, and local  $L^2$-convergence to initial data. We also get \eqref{LEI.tdv} with the time interval $I$ replaced by $[t_0,t]$ and an additional term $\int |\td v|^2 \xi (x,t)\,dx$ in the left side.

We have shown that $(\td v, \td p)$ is a local energy solution on $\R^3 \times I$ with initial data $\td v|_{t=t_0} = v(t_0)$. 

\bigskip

\noindent\texttt{Step 3.} To extend to a time-global local energy solution.

\medskip

We first prove that the combined solution
\[
u = v1_{[0,t_0]}+ \td v 1_{I},\quad  q = p1_{[0,t_0]}+ \td p 1_{I}
\]
is a local energy solution on the extended time interval $[0,T_1]=[0,t_0+S]$. It is obvious that $u$ and $q$ are bounded in $\cE_{T_1}$ and $L^\frac 32_\loc ([0,T_1]\times \R^3)$, respectively and $q$ satisfies the decomposition at each point $x_0\in \R^3$. Since we have for any $\zeta \in C_c^\infty([t_0,T_1)\times \R^3;\R^3)$
\[
\int_{t_0}^{T_1} -(\td v, \pa_t \zeta) + (\na \td v, \na \zeta) + (\td v, (\td v\cdot \na) \zeta) + (\td p, \div \zeta) dt = (\td v,\zeta)(t_0) = (v,\zeta)(t_0),
\]
and for any $\zeta \in C_c^\infty((0,t_0]\times \R^3;\R^3)$
\[
\int_{0}^{t_0} -(v, \pa_t \zeta) + (\na v, \na \zeta)+ ( v, ( v\cdot \na) \zeta) + ( p, \div \zeta) dt = -(v,\zeta)(t_0),
\]
from the weak continuity of $\td v$ at $t_0$ from the right and that of $v$ at $t_0$, we can prove that $(u,p)$ satisfies \eqref{NS} in the distribution sense: For any $\zeta\in C_c^\infty((0,T_1)\times \R^3;\R^3)$
\[
\int_{0}^{T_1} -(u, \pa_t \zeta) + (\na u, \na \zeta)+ ( u, ( u\cdot \na) \zeta) + ( q, \div \zeta) dt = 0.
\]

Also, since we already have local $L^2$-weak continuity of $u$ on $[0,T_1]\setminus \{t_0\}$, it is enough to check it at $t_0$; for any $\ph\in L^2(\R^3)$ with a compact support, 
\[
\lim_{t\to t_0^-} (u, \ph)(t)
=\lim_{t\to t_0^-} (v, \ph)(t)
=(v,\ph)(t_0)
=\lim_{t\to t_0^+} (\td v, \ph)(t)
=\lim_{t\to t_0^+} (u, \ph)(t).
\]

Finally, we prove the local energy inequality \eqref{LEI}. Indeed, for any $t\in (0,t_0]$, the inequality follows from the one of $v$. For $t\in (t_0,T_1)$, we add the inequality of $v$ in $[0,t_0]$ to the one of $\td v$ in  $[t_0,t]$ to get, for any non-negative $\xi \in C_c^\infty((0,T_1)\times \R^3)$,
\EQN{
\int |u|^2\xi &(t) dx
+2\int_0^t\! \int |\na u|^2\xi dxds
\\
&= \int |\td v|^2\xi(t) dx
+2\int_0^{t_0} \int |\na v|^2\xi dxds
+2\int_{t_0}^t\!\int |\na \td v|^2\xi dxds\\
&\leq  
\int_0^{t_0}\int |v|^2 (\pa_s \xi + \De \xi) +(|v|^2 +2p)(v\cdot \na)\xi dxds\\
&\qquad +\int_{t_0}^t\!\int |\td v|^2 (\pa_s \xi + \De \xi) +(|\td v|^2 +2\td p)(\td v\cdot \na)\xi dxds\\
&=\int_0^t\!\int |u|^2 (\pa_s \xi + \De \xi) +(|u|^2 +2q)(u\cdot \na)\xi dxds.
}

Therefore, $(u,q)$ is a local energy solution on $[0,T_1]$ and is an extension of $(v,p)$.

Then, by Lemma \ref{dot.E.3.2} and the proof of Corollary \ref{w.E4}, we can find $t_1\in (t_0 + \frac 78 S, t_0 + S)$ such that $ u(t_1)-V(t_1) \in E^4$. Repeating the above argument with new initial time $t_1$, we can get a  local energy solution in $[0,t_1+S)$. Iterating this process, we get
a local energy solution global in time. Note that $\norm{V}_{L^\infty([t_1,\infty)\times \R^3)}\leq \norm{V}_{L^\infty([t_0,\infty)\times \R^3)}$ whenever $t_1 >t_0$, so that on each step, we can extend the time interval for the existence by at least $\frac78 S$.  
\end{proof}

%
%
%
\section{Perturbations of global solutions with no spatial oscillation decay}
\label{sec6}

As mentioned in the introduction, there are many known non-decaying flows like constant flows, spatially periodic flows (flows on torus) and \emph{two-and-a-half dimensional flows}. The last two do not have oscillation decay in general. We do not have a general existence theory for initial data with no oscillation decay. However, the method of this paper can be used to construct perturbations of global solutions with no spatial oscillation decay.  The perturbation of a constant flow is already covered by Theorem \ref{global.ex}. The 
perturbation of spatially periodic flows and two-and-a-half dimensional flows are covered by the following theorem, which does not assume spatial decay or spatial oscillation decay of initial data.

\begin{theorem}\label{theorem2}
Let $V(x,t)$ be a global in time local energy solution of \eqref{NS} with
\[
V \in L^\infty(0,\I; L^q_\uloc), \quad V|_{t=0}=V_0 \in L^q_{\uloc,\si},
\]
for some $q>3$. Then for any $w_0 \in E^2_\si$, there is a global-in-time local energy solution $v$  of \eqref{NS} with initial data 
$v_0= V_0 + w_0$.
\end{theorem}

\begin{proof} We may assume $3<q<\infty$. Let $P$ be an associated pressure of $V$. Let $w = v-V$ and $q=p-P$. If $(v,p)$ is a solution of \eqref{NS}, then $(w,q)$ should satisfy the perturbed equation
\begin{equation}
\begin{cases}
\pa_t w -\De w + (V+w)\cdot \na w + w \cdot \nb V + \na q = 0  ,\quad
\div w =0 \\
w|_{t=0}=w_0 ,
\end{cases}
\end{equation}
which is \eqref{PNS} without the source term $V \cdot \nb V$. As a result, we don't need the spatial decay of $\nb V$, the strong local energy inequality \eqref{SLEI-v}, or the spatial decay estimate \eqref{pre.decay.est00} with $\nb V$. Hence, the proof is much easier.

Since $v_0 \in L^2_\uloc$, a local energy solution $v$ to \eqref{NS} exists on the time interval $[0,T]$ for some $T>0$ by Theorem \ref{loc.ex}.
Using Lemma \ref{th:LEI.w.0}, we have the local energy estimate for $w$
\EQ{ \label{LEI,w2}
\int & |w|^2(x,t) \xi(x)\, dx
+ 2\int_{0}^t\!\int |\na w|^2 \xi\, dxds \\ 
& \leq \int  |w_0|^2 \xi(x)\, dx
+\int_{0}^t\!\int |w|^2 (\De \xi +v\cdot \na\xi )\,dxds\\
&\quad + \int_{0}^t\!\int 2 q_{x_0} w \cdot \na \xi\, dxds
+ \int_{0}^t\!\int  2V\cdot (w \cdot \na )(w\xi)\, dxds,
}
for any $\xi\in C^\infty_c(\R^3)$, $\xi \ge 0$. Here ${q}_{x_0}$ is defined by
\EQN{
q_{x_0}(x,t) =& -\frac 13 (|w(x,t)|^2+2w\cdot V) + \pv \int_{B(x_0,2)} K_{ij}(x-y)(w_iw_j+V_iw_j+w_iV_j)(y,t)dy\\
&+\int_{B(x_0,2)^c} (K_{ij}(x-y)-K_{ij}(x_0-y))(w_iw_j+V_iw_j+w_iV_j)(y,t)dy,
}
where $K_{ij}(x) = \pa_{ij} \frac 1{4\pi |x|}$. Let $\phi_{x_0}$ and $\chi_R$ be defined as in \eqref{xi.def}.
 We first derive an a priori bound from \eqref{LEI,w2} taking $\xi = \phi_{x_0}^2$ and taking sup over $x_0 \in \R^3$ using $q>3$ (compare \eqref{eq6.4} below for the last term of \eqref{LEI,w2})
 \EQ{\label{eq6.3}
\sup_{0<t<T}\norm{w(\cdot, t)}_{L^2_\uloc}^2 
+ \norm{\na w}_{U^{2,2}_T}^2  + \norm{w}_{U^{3,3}_T} \leq A,
}
where $A=A(T,\norm{w_0}_{L^2_\uloc}, q, \norm{V}_{L^\infty L^q_{\uloc}})$.
Next, by the proof of \cite[Section 2]{KS} with $\xi= \phi_{x_0}^2\chi_R^2$, we can prove a spatial decay estimate (easier than  \eqref{pre.decay.est00})
\EQ{\label{decay.w.6}
\sup_{0<t<T}\norm{w(\cdot, t)\chi_R}_{L^2_\uloc}^2 + \norm{|\na w|\chi_R}_{U^{2,2}_T}^2 
\leq C\left(\norm{w_0\chi_R}_{L^2_\uloc}^2 + R^{-\frac 23}\right), 
}
where $C=C(T, A, q,\norm{V}_{L^\infty L^q_{\uloc}})$.
Indeed, all terms in \eqref{LEI,w2} except the last one can be estimated in the same way. For the last term,
\begin{align*}
\int_0^T\int &V (w\cdot \na) (w\phi_{x_0}^2\chi_R^2) dxdt\\
&\lec  
\int_0^T\int |V||w|(|\na w|\phi_{x_0}^2\chi_R^2
+ |w|\phi_{x_0}\chi_R^2 + \frac 1R |w|\phi_{x_0}^2) dxdt\\
&\lec_{A,T,q} \norm{V}_{L^\infty(0,\infty;L^q_\uloc)}\left[\norm{w\chi_R}_{U^{2,\left( \frac 12-\frac 1q\right)^{-1}}_T}\norm{|\na w|\chi_R}_{U^{2,2}_T} + \norm{w\chi_R}_{U^{3,3}_T}^2 + \frac 1R 
\right].
\end{align*}
Then, we use the Gagliardo-Nirenberg interpolation inequality to get
\begin{align*}
\norm{w\chi_R}_{U^{2,\left( \frac 12-\frac 1q\right)^{-1}}_T}
&\lec\norm{\na (w\chi_R)}_{U^{2,2}_T}^\frac 3q \norm{w\chi_R}_{U^{2,2}_T}^{1-\frac 3q} +\norm{w\chi_R}_{U^{2,2}_T}\\
&\lec \norm{|\na w|\chi_R}_{U^{2,2}_T}^\frac 3q\norm{w\chi_R}_{U^{2,2}_T}^{1-\frac 3q}+\norm{w\chi_R}_{U^{2,2}_T} + \frac {C_q(A,T)}{R^\frac3q},
\end{align*}
and hence (using $q>3$ to get a small constant)
\EQ{\label{eq6.4}
\int_0^T\int V (w\cdot \na) (w\phi_{x_0}^2\chi_R^2) dxdt
\leq  \frac 1{99} \norm{|\na w|\chi_R}_{U^{2,2}_T}^2 + C_q(A,T) \left(\norm{w\chi_R}_{U^{3,3}_T}^2 + \frac 1{R^\frac3q}\right).
}
This is enough to complete the proof for \eqref{decay.w.6}. Finally, as in Corollary \ref{w.E4}, it implies 
\begin{align}\label{w2Ep}
w(t)\in E^p(\R^3), \quad \text{for almost all }t\in (0,T]
\end{align} for any $3\leq p\leq  6$. 

Now, we repeat the extension argument in Section \ref{global.sec} with the replacement of the heat equation solution by the time-global solution $V$ given in Theorem \ref{theorem2}. Assume that a local energy solution $(v,p)$ to \eqref{NS} for initial data $v_0\in L^2_\uloc(\R^2)$ exists on $[0,T_0]$, $T_0\in (0,\infty)$. Then, by \eqref{w2Ep}, we can find $t_0 \in (0,T_0)$, arbitrarily close to $T_0$, such that $w(t_0) = W_0 +h_0$ where $W_0 \in C_{c,\si}^\infty(\R^3)$ and $h_0 \in E^4(\R^4)$ with $\norm{h_0}_{L^4_\uloc}<\de$. The construction of a local energy solution $(\td v, \td p)$ after time $t_0$ proceeds as follows. We decompose the solution
\[
\td v  = V + h + W, \quad \td p = p_V + p_h + p_W,
\] 
where $V$ is the given solution with pressure $p_V$, $(h,p_h)$ solves
\begin{align}\label{eqn.h2}
\begin{cases}
\pa_t h -\De h + \na p_h = -(V+h)\cdot \na h - (h\cdot \na)V\\
\div h=0, \quad h|_{t=t_0} = h_0,
\end{cases}
\end{align}
and $(W,p_W)$ satisfies \eqref{eqn.td.w} with the given solution $V$. The only difference with \eqref{eqn.h} is that \eqref{eqn.h2} excludes the term $(V\cdot \na)V$. With the interior regularity (see e.g.~\cite[Theorem A1]{LuoTsai})
\[
\norm{V}_{L^\infty(\R^3 \times (t_0,\infty))} \le C(t_0, \norm{V}_{L^\infty(0,\I; L^q_\uloc)}),
\]
(we need the strict inequality $q>3$ for this uniform estimate),
the rest of the proof is the same as in Section \ref{global.sec}. 
\end{proof}

\section*{Acknowledgments}
The research of both Kwon and Tsai was partially supported by NSERC grant 261356-13.

Hyunju Kwon, 
 Department of Mathematics, University of British Columbia, Vancouver, BC V6T 1Z2, Canada

Current address: 
Department of Mathematics, Institute for Advanced Study, 
Princeton, NJ 08540, USA;
 e-mail: hkwon@ias.edu

 \medskip
 
 Tai-Peng Tsai, Department of Mathematics, University of British Columbia, 
 Vancouver, BC V6T 1Z2, Canada;
 e-mail: ttsai@math.ubc.ca
 

\begin{thebibliography}{XX}

\bibitem{Barraza} Barraza, O., Self-similar solutions in weak $L^p$-spaces of the Navier-Stokes equations. Rev. Mat. Iberoamericana 12 (1996), 411-439.

\bibitem{BCD} Bahouri, H.. Chemin, J.-Y., Danchin, R., Fourier analysis and nonlinear partial differential equations. Grundlehren der Mathematischen Wissenschaften [Fundamental Principles of Mathematical Sciences], 343. Springer, Heidelberg, 2011. 

\bibitem{CKN} Caffarelli, L., Kohn, R. and Nirenberg, L., Partial regularity of suitable weak solutions of the Navier-Stokes equations. Comm. Pure Appl. Math. 35 (1982), no. 6, 771-831.

\bibitem{Calderon} Calder\'on, Calixto P. 
Existence of weak solutions for the Navier-Stokes equations with initial data in $L^p$. Trans. Amer. Math. Soc. 318 (1990), no. 1, 179-200. 

\bibitem{CP} Cannone, M. and Planchon, F., Self-similar solutions for Navier-Stokes equations in $\R^3$. Comm. Partial Differential Equations 21 (1996), no. 1-2, 179-193. 


\bibitem{FJR}
Fabes, E.~B., Jones, B.~F., and Rivi{\`e}re, N.~M.,\emph{The initial value
  problem for the {N}avier-{S}tokes equations with data in {$L^{p}$}}, Arch.
  Rational Mech. Anal. \textbf{45} (1972), 222--240.

\bibitem{Gallagher}
Gallagher, I., The tridimensional Navier-Stokes equations with almost bidimensional data: stability, uniqueness, and life span. Internat. Math. Res. Notices 1997, no. 18, 919-935.


\bibitem{Gi86} Giga, Y., Solutions for semilinear parabolic equations in $L^p$ and regularity of weak solutions of the Navier-Stokes system, J. Differential Equations 62 (1986), no. 2,
186–212. 

 \bibitem{GIM} Giga, Y., Inui, K. and Matsui, S., On the Cauchy problem for the Navier-Stokes equations with nondecaying initial data. Advances in fluid dynamics, 27-68, Quad. Mat., 4, Dept. Math., Seconda Univ. Napoli, Caserta, 1999.
 
\bibitem{GiMi} Giga, Y. and Miyakawa, T., Navier-Stokes flows in $\R^3$ with measures as initial vorticity and the Morrey spaces, Comm. Partial Differential Equations 14 (1989), 577-618.

\bibitem{Hopf} Hopf, E.,
{{\"U}ber die {A}nfangswertaufgabe f{\" u}r die
  hydrodynamischen {G}rundgleichungen}, Math. Nachr. \textbf{4} (1951),
  213--231.


\bibitem{JiaSverak-minimal} Jia, H. and \v Sver\'ak, V., Minimal $L^3$-initial data for potential Navier-Stokes singularities. SIAM J. Math. Anal. 45 (2013), no. 3, 1448-1459.

\bibitem{JiaSverak} Jia, H. and \v Sver\'ak, V., Local-in-space estimates near initial time for weak solutions of the Navier-Stokes equations and forward self-similar solutions. Invent. Math. 196 (2014), no. 1, 233-265.


\bibitem{KMT} Kang, K., Miura, H. and Tsai, T.-P., 
Short time regularity of Navier-Stokes flows with locally $L^3$ initial data and applications, 
Int. Math. Res. Not., to appear, arXiv:1812.10509. 


\bibitem{Kato84} Kato, T.,
{Strong {$L^{p}$}-solutions of the {N}avier-{S}tokes equation
  in {${\bf R}^{m}$}, with applications to weak solutions}, Math. Z.
  \textbf{187} (1984), no.~4, 471--480.
  
\bibitem{Kato} Kato, T., Strong solutions of the Navier-Stokes equation in Morrey spaces. Bol. Soc. Brasil. Mat. (N.S.) 22 (1992), no. 2, 127-155.

\bibitem{KS} Kikuchi, N. and Seregin, G., Weak solutions to the Cauchy problem for the Navier-Stokes equations satisfying the local energy inequality. Nonlinear equations and spectral theory, 141-164, Amer. Math. Soc. Transl. Ser. 2, 220, Amer. Math. Soc., Providence, RI, 2007.

\bibitem{Koch-Tataru}
Koch, H. and Tataru, D., Well-posedness for the Navier-Stokes equations.
Adv. Math. 157 (1), 22--35 (2001)

%

\bibitem{KoYa} Kozono, H. and Yamazaki, M., {Semilinear heat equations and the navier-stokes equation with distributions in new function spaces as initial data}, Comm. Partial Differential Equations, 19 (1994), 959-1014.   

  
\bibitem{LR} Lemari\'e-Rieusset, P. G., \emph{Recent developments in the Navier-Stokes problem.} Chapman Hall/CRC Research Notes in Mathematics, 431. Chapman Hall/CRC, Boca Raton, FL, 2002.


\bibitem{LR-Morrey} Lemari\'e-Rieusset, P. G., The Navier-Stokes equations in the critical Morrey-Campanato space. Rev. Mat. Iberoam. 23 (2007), no. 3, 897-930.

\bibitem{LR2} Lemari\'e-Rieusset, P. G.,  \emph{The Navier-Stokes problem in the 21st century}. CRC Press, Boca Raton, FL, 2016.
%
\bibitem{leray} Leray, J., Sur le mouvement d'un liquide visqueux emplissant l'espace. (French) Acta Math. 63 (1934), no. 1, 193-248. 


\bibitem{LuoTsai}
Y. Luo and T.-P. Tsai, Regularity criteria in weak $L^3$ for 3D incompressible Navier-Stokes equations, Funkcialaj Ekvacioj 58 (2015) 387-404. 


\bibitem{MaMiPr}  
Maekawa, Y., Miura, H., and Prange, C.,
Local energy weak solutions for the Navier-Stokes equations in the half-space,
Comm. Math. Phys. 367 (2019), no. 2, 517-580.


\bibitem{MT} Maekawa, Y.; Terasawa, Y., The Navier-Stokes equations with initial data in uniformly local $L^p$ spaces. Differential Integral Equations 19 (2006), no. 4, 369-400. 

\bibitem{MaBe}
Majda, A.~J., and Bertozzi, A.~L.,
\newblock {\em Vorticity and incompressible flow}, volume~27 of {\em Cambridge
  Texts in Applied Mathematics}.
\newblock Cambridge University Press, Cambridge, 2002.

\bibitem{MaSe} Maremonti, P., Shimizu, S.,
Global existence of solutions to 2-D Navier–Stokes flow with non-decaying initial data in half-plane,
J. Differential Equations 265 (2018) 5352--5383. Errata: ibid
266 (2019), no. 7, 3925-3926.


  
\end{thebibliography}
\end{document}